\documentclass[11pt,a4paper]{article}

\usepackage{amsmath,amssymb}
\usepackage{lineno}
\usepackage{dsfont}
\usepackage{amsfonts}
\usepackage{amsthm}
\usepackage{mathtools}
\usepackage{multirow}
\usepackage{physics}
\usepackage{graphicx}
\usepackage{caption}
\usepackage{subcaption}
\usepackage{listings}
\usepackage{relsize}
\usepackage{bigints}
\usepackage[square,sort,comma,numbers]{natbib}
\usepackage[export]{adjustbox}
\usepackage[normalem]{ulem}
\usepackage{tikz}
\usetikzlibrary{matrix,cd,arrows,shapes.geometric}

\numberwithin{equation}{section}
\theoremstyle{plain}
\newtheorem{theorem}{Theorem}[section]
\newtheorem{lemma}[theorem]{Lemma}
\newtheorem{proposition}[theorem]{Proposition}

\theoremstyle{definition}
\newtheorem{definition}[theorem]{Definition}
\newtheorem{assumption}[theorem]{Assumption}
\theoremstyle{remark}
\newtheorem{remark}[theorem]{Remark}
\newcommand{\mT}{\mathcal T}

\newcommand{\ve}{\varepsilon}
\newcommand{\la}{\langle}
\newcommand{\ra}{\rangle}

\begin{document}

\title{\huge Homogenization of the Richards equation in porous media: Multiscale modelling of  water transport in vegetated soil}
\author{{Andrew Mair$\,^{\mathrm{a}}$,    Mariya Ptashnyk$\,^{\mathrm{b}}$}\medskip\\
\small $^{\mathrm{a}}$  Department of Conservation of Natural Resources, NEIKER, Derio, Basque Country, Spain \\
\small $^{\mathrm{b}}$ Maxwell Institute for Mathematical Sciences, Department of Mathematics, School of\\  \small  Mathematical and Computer Sciences, Heriot-Watt University, Edinburgh, Scotland, UK}
\date{}

\maketitle

\vspace*{-0.4in}\begin{abstract}
In this paper we consider the multiscale modelling of water transport in vegetated soil. In the microscopic model we distinguish between subdomains of soil and plant tissue, and use the Richards equation to model the water transport through each. Water uptake is incorporated by means of a boundary condition on the surface between root tissue and soil. Assuming a simplified root system architecture, which gives a cylindrical microstructure to the domain, the two-scale convergence  and periodic unfolding methods are applied to rigorously derive a macroscopic model for water transport in vegetated soil. The degeneracy of the Richards equation and the dependence of root tissue permeability on the small parameter introduce considerable challenges in deriving macroscopic equations, especially in proving strong convergence. The variable-doubling method is used to prove the uniqueness of solutions to the model, and also to show strong two-scale convergence in the nonlinear terms of the equation for water transport through root tissue.
\end{abstract}

{\small \noindent{\it Keywords:}
Richards equation, degenerate parabolic equations, variable-doubling, dual-porosity,  homogenization, two-scale convergence and periodic unfolding
}

{\small \noindent{\it MSC codes:}
 35B27, 35K51, 35K65
}

\section{Introduction}
Mathematical models for water transport in soil are used to inform practice within different environmental contexts such as calculating crop yields~\citep{holzworth2014apsim}, monitoring flood risk~\citep{el1998integrated}, or  assessing hill slope stability~\citep{endrizzi2014geotop}.  Moreover, predictions from  models are   becoming even more important when  developing strategies to deal with the increasing impact of climate change.
The common model
for water transport through soil is  the Richards equation~\cite{richards1931capillary}, which describes water flow in unsaturated porous media. The effects of root-soil interactions tend to be accounted for by adding a sink term that models root water uptake~\cite{simunek2005hydrus}. These macroscopic sink terms  use the distribution of soil water content and root length density to determine the contribution of each soil zone to plant transpiration~\cite{vsimuunek2009modeling}, but do not explicitly account for water flow within the root system.   Alternatively, there exist models at the microscopic scale. Some of these represent the hydraulic architecture of the root system as a network of connected nodes, with axial and radial hydraulic conductivity values determining the rates of inter-root and root-soil water transport~\cite{doussan1998modelling,javaux2008use, couvreur2012simple}. Other microscopic models explicitly consider water flow through subdomains of soil and root tissue~\cite{arbogast1993simulation, doussan2006water, roose2004model}.
 However, the computational expense for numerical simulation of these microscopic models becomes prohibitively high as root system complexity and scale increases. Moreover,  these models do not consider    the influence of roots on the hydraulic properties of soil.  Dual-porosity models~\citep{gerke1993dual}, that can be rigorously derived from the microscopic description  via homogenization techniques~\cite{Arbogast_1990,  Bourgeat_1996, hornung1996homogenization}, incorporate  preferential flow through soil with heterogeneous hydraulic properties, but do not relate the pattern of this heterogeneity to the  root system architecture. Models for the preferential flow of soil water that is induced by the presence of a root
system were developed in~\cite{mair2023can, mair2022model}, but with a phenomenological formulation of the influence of roots on water transport and uptake.

Thus, the main aim of this paper is to  formulate a microscopic model that explicitly distinguishes between water flow through subdomains of soil and  root branches and  apply homogenization techniques to derive  a macroscopic model  that can be efficiently simulated while  accounting for water transport through plant roots and the effect of roots on the hydraulic properties of the soil.  In the  microscopic model, the Richards equation, with different permeability and water retention functions, is used to model water flow through the soil and root subdomains. Water uptake is described by the boundary conditions at the root surface. Pressure gradients within the root tissue are driven by transpiration,  incorporated through a boundary condition at the root crown. The difference between the properties of the bulk soil and the rhizosphere,  a narrow region of soil  around root branches, is reflected through the explicit spatial dependence  in the water retention and hydraulic conductivity functions.   In this work we consider periodically arranged root branches. For more general root systems, this simple geometry can be considered locally and, with a suitable transformation, a macroscopic model can  be derived using the locally periodic two-scale convergence and unfolding methods, \cite{ptashnyk2015, ptashnyk2013}.

The Richards equation, with its nonlinearities in the time derivative and elliptic operator, belongs to a larger class of degenerate parabolic equations.  Well-posedness results for these types of equations  and systems have been obtained by many authors, see e.g.~\cite{alt1983quasilinear, Benilan_1996, DiBenedetto_1987,  Jaeger_Kacur_1995,    Kacur_2009, Schweizer_2007}. Existence of strong solutions to the Richards equation was addressed  in~\cite{Merz_Rybka_2010}
and viscosity solutions for a general class of nonlinear degenerate parabolic equations were studied in~\cite{Kim_Pozar_2013}. In~\cite{otto1996l1, Knabner_Otto_2000} the uniqueness of solutions for degenerate parabolic equations and systems was shown using the variable-doubling method~\cite{Kruzkov_1970}, and  proving the $L^1$-contraction principle.

There are several results on  multiscale analysis  for the Richards equation and  degenerate parabolic equations in general. The homogenization of nonlinear parabolic equations with Dirichlet boundary conditions in a domain periodically perforated by small holes was studied in~\cite{Nandakumaran_2002}.  In~\cite{Amaziane_2013,  Benes_2019, cao_2014} the two-scale convergence method was applied to nonlinear equations with oscillating coefficients modelling the two-phase flow in a porous medium and in~\cite{chen_2005}  the  asymptotic analysis and error estimates were considered.
In~\cite{Jaeger_2021} the homogenization of the  Richards equation, in a medium with deterministic almost periodic microstructure of coarse and fine impermeable soft inclusions, is obtained using sigma-convergence method. An upscaling of the Richards equation to describe the flow in fractured porous media was considered in~\cite{List_2020}. Furthermore, in~\cite{Mikelic_2004} homogenization techniques have  been applied to dual-porosity models with  the flow  through porous particles and the inter-particle space described by the Richards equation.

In our microscopic model we consider a cylindrical microstructure representing the root branches where permeability of the root tissue is proportional to root branch radius and, hence, depends on the small parameter $\ve$.  Mutilscale analysis for parabolic and elliptic equations defined in domains with cylindrical microstructures was considered in~\cite{Melnyk_2007, Murat_2017, Neuss-Radu_1996}. Similar to~\cite{Mikelic_2004}, the dependence of the  permeability of the particles  on the small parameter requires  novel analytical approaches, compared to the existing multiscale analysis results for the Richards equations.  In particular, the $\ve$-dependence of the permeability tensor  means that for the proof of the strong convergence, required to pass to the limit in the nonlinear terms, a standard compactness argument cannot be used. In~\cite{Mikelic_2004} the convergence of the nonlinear terms was shown using a specific form of test function, in conjunction with the unfolding operator, also referred to as the dilation operator and periodic modulation method~\cite{Arbogast_1990, Bourgeat_1996}. Here, we apply the time-doubling method to show the equicontinuity  of the unfolded sequence of solutions to the microscopic problem  and  to prove the strong two-scale convergence. This is a novel application  of the variable-doubling method in homogenization theory, which will allow the derivation of macroscopic equations for other degenerate parabolic problems. We would also like to remark that the ideas used to show the convergence results in~\cite{Mikelic_2004} cannot be applied to our problem due to different  scaling in the permeability function and different boundary conditions on the surfaces of the microstructure.
 More specifically, in~\cite{Mikelic_2004} the continuity of  the pressure on the surface of the microstructure, i.e.~boundary between pore space and porous pellets (aggregates) of soil particles (corresponding in our model to the boundaries between soil and root tissue) is considered. In addition, there is an $\ve^2$-scaling in the flux within the pellets. These conditions allow the formulation of a problem for the unfolded function (dilation operator) of the pressure in the pellets, satisfying the Richards equation with differential operator with respect to the microscopic variables and with the Dirichlet boundary conditions, given by the pressure in the pore space. Then a specific test function is considered to prove the strong convergence result for the pressure in pellets. In our model, due to experimental measurements stating that root permeability is proportional to the root radius, we have $\ve$-scaling for the flux across the root radius. Also, we assume uptake of water by the roots is driven by the difference in pressure head between the root tissue and the soil, which leads to a potential discontinuity of pressure at the root-soil boundary. The resulting structure
of the equations does not allow us to formulate a problem for the unfolded pressure with respect to the microscopic variables.  Therefore, we cannot use  a similar choice of specific test function as in~\cite{Mikelic_2004} to obtain the estimates needed for the proof of strong convergence. Hence, a different approach is required.

The paper is organised as follows. In section~\ref{section:model_formulation} we formulate the microscopic model for the water flow in vegetated soil. Section~\ref{section:existence_uniqueness} is devoted to the analysis of the microscopic problem.  The  proof of the existence result follows the same ideas as in~\cite{alt1983quasilinear}, and we include only the main steps of the proof, and the estimates involving boundary conditions at the soil-root interface, which are not considered in~\cite{alt1983quasilinear}. We also present the proof of  the convergence for the nonlinear functions in the time derivatives, which  is simpler for our model compared to more general setting  analysed in~\cite{alt1983quasilinear}. The homogenization results and the derivation of the  macroscopic model are  presented in section~\ref{section:macro_model}. Typical examples for the nonlinear hydraulic conductivity  and water retention functions are given in the Appendix.

\section{Formulation of microscopic model}\label{section:model_formulation}

For the microscopic description of water flow in a vegetated soil,  we consider a simplified root system architecture where the root branches are straight, vertically oriented, of uniform radius, and periodically distributed throughout the domain.

Consider  $\Omega = (0, L_1)\times(0, L_2)\times(-L_3, 0)$ representing a section of vegetated soil, with $\Gamma_0 = \partial\Omega \cap \{ x_3 = 0\}$ and $\Gamma_{L_3} = \partial\Omega \cap \{ x_3 = -L_3\}$ . The regions occupied by root branches~$P^\ve$, the rhizosphere soil around each root branch~$R^\ve$, the bulk soil~$B^\ve$, and the combined total soil subdomain~$S^\ve$ are  defined as $B^\ve= \Omega \setminus \big( \overline P^\ve \cup \overline R^\ve\big)$, $S^{\varepsilon} = {\rm Int}\big(\overline R^{\varepsilon}\cup \overline B^{\varepsilon}\big)$, and
$$
P^\ve = \bigcup_{\xi \in \Xi_\ve} \ve (Y_P +\xi) \times (-L_3, 0), \quad  R^\ve = \bigcup_{\xi \in \Xi_\ve} \ve (Y_R \setminus \overline Y_P +\xi) \times (-L_3, 0).
$$
Here $\Xi_\ve = \{\xi \in \mathbb Z^2 \, : \,  \ve (\overline Y +\xi) \subset  (0, L_1)\times(0, L_2) \}
$,  the unit cell $Y=(0,1)^2$ and  $\overline Y_P \subset \overline Y_R \subset Y$ with boundaries $\Gamma_P$ and $\Gamma_R$ respectively, $\hat R= Y_R \setminus \overline Y_P$, and $\hat B= Y \setminus \overline Y_R$, see Figure~\ref{figure:entire_domain}.
The lateral surfaces of the  microstructure, given by the  root branches,  and  the boundaries between the rhizosphere and bulk soil are denoted by
$$
\Gamma_P^\ve=  \bigcup_{\xi \in \Xi_\ve} \ve (\Gamma_P  +\xi) \times (-L_3, 0)\quad \text{ and } \quad\Gamma_R^\ve=  \bigcup_{\xi \in \Xi_\ve} \ve (\Gamma_R  +\xi) \times (-L_3, 0),
$$
and, with~$J=B,~P,~R,~S$, define
$$
		\Gamma_{J, 0}^\varepsilon = \{x\in\partial\Omega\cap  \partial J^\varepsilon: x_3 = 0\},  \qquad
		\Gamma_{J, L_3}^\varepsilon  = \{x\in\partial\Omega\cap  \partial J^\varepsilon: x_3 = -L_3\}.
$$
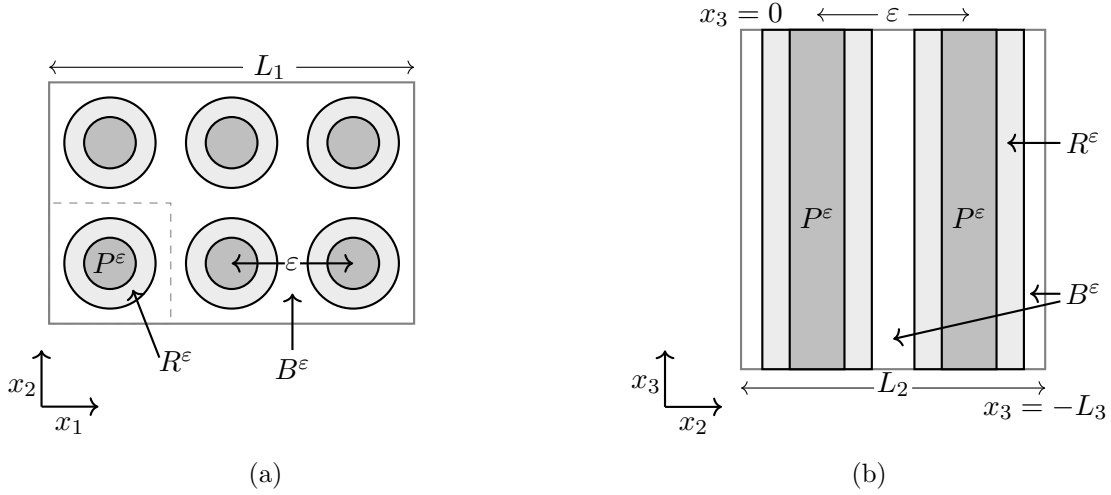
\begin{figure}
	\centering
	\begin{subfigure}{0.45\textwidth}
		\begin{tikzpicture}
			\draw [->, thick] (0,-1) -- (0.75,-1);
			\draw [->, thick] (0,-1) -- (0,-0.25);
			\node at (-0.2,-0.75) {$x_2$};
			\node at (0.37,-1.25) {$x_1$};
			\draw[gray, thick] (0.1,0.1) rectangle (4.9,3.3);
			\draw[gray, dashed] (0.1,0.1) rectangle (1.7,1.7);
			\filldraw[fill=gray!15, thick] (0.9, 0.9) circle (0.6);
			\filldraw[fill=gray!50, thick] (0.9, 0.9) circle (0.34);
			\filldraw[fill=gray!15, thick] (0.9, 2.5) circle (0.6);
			\filldraw[fill=gray!50, thick] (0.9, 2.5) circle (0.34);
			\filldraw[fill=gray!15, thick] (2.5, 0.9) circle (0.6);
			\filldraw[fill=gray!50, thick] (2.5, 0.9) circle (0.34);
			\filldraw[fill=gray!15, thick] (2.5, 2.5) circle (0.6);
			\filldraw[fill=gray!50, thick] (2.5, 2.5) circle (0.34);
			\filldraw[fill=gray!15, thick] (4.1, 0.9) circle (0.6);
			\filldraw[fill=gray!50, thick] (4.1, 0.9) circle (0.34);
			\filldraw[fill=gray!15, thick] (4.1, 2.5) circle (0.6);
			\filldraw[fill=gray!50, thick] (4.1, 2.5) circle (0.34);
			\draw [<-, thick] (2.5,0.9) -- (3.2,0.9);
			\node at (3.3,0.9) {$\varepsilon$};
			\draw [->, thick] (3.4,0.9) -- (4.1,0.9);
			\node at (0.9,0.9) {$P^\varepsilon$};
			\node at (1.77,-0.4) {$R^\varepsilon$};
			\draw [->, thick] (1.55,-0.35) -- (1.2,0.55);
			\node at (3.3,-0.5) {$B^\varepsilon$};
			\draw [->, thick] (3.3,-0.25) -- (3.3,0.5);
			\draw [<-] (0.1,3.5) -- (2.65,3.5);
			\node at (3,3.5) {$L_1$};
			\draw [->] (3.35,3.5) -- (4.9,3.5);
		\end{tikzpicture}
		\caption{}
	\end{subfigure}
	\qquad
	\begin{subfigure}{0.45\textwidth}
		\centering
		\begin{tikzpicture}
			\draw[gray, thick] (1,0.5) rectangle (5,5);
			\filldraw[fill=gray!15, thick] (1.28,0.5) rectangle (2.72,5);
			\filldraw[fill=gray!50, thick] (1.64,0.5) rectangle (2.36,5);
			\filldraw[fill=gray!15, thick] (3.28,0.5) rectangle (4.72,5);
			\filldraw[fill=gray!50, thick] (3.64,0.5) rectangle (4.36,5);
			\draw [<-] (1,0.25) -- (2.75,0.25);
			\node at (3,0.3) {$L_2$};
			\draw [->] (3.25,0.25) -- (5,0.25);
			\node at (1,5.2) {$x_3=0$};
			\node at (1,0.0) {$x_3=-L_3$};
			\node at (2,2.5) {$P^\varepsilon$};
			\node at (4,2.5) {$P^\varepsilon$};
			\draw [<-] (2,5.2) -- (2.8,5.2);
			\node at (3,5.2) {$\varepsilon$};
			\draw [->] (3.2,5.2) -- (4,5.2);
			\node at (5.5,3.5) {$R^\varepsilon$};
			\draw [->, thick] (5.2,3.5) -- (4.5,3.5);
			\node at (5.5,1.5) {$B^\varepsilon$};
			\draw [->, thick] (5.2,1.5) -- (4.8,1.5);
			\draw [->, thick] (5.2,1.4) -- (3,0.9);
		\end{tikzpicture}
		\caption{}
	\end{subfigure}
	\caption{An illustration of the domain~$\Omega$ comprised of bulk soil~$B^\varepsilon$,  rhizosphere soil~$R^\varepsilon$. and root tissue~$P^\varepsilon$; (a)  from above; (b)  cross-section of the~$x_2-x_3$ plane at~$x_1 = {L_1}/{2}$.}
	\label{figure:entire_domain}
\end{figure}
The water transport in unsaturated soil is modelled by the Richards equation
\begin{equation}
	\begin{aligned}\label{richards_soil}
		\partial_t\theta_{S}^\varepsilon(x,h_{S}^\varepsilon) - \nabla\cdot\big(K_S^\varepsilon(x,h_{S}^\varepsilon)(\nabla h_{S}^\varepsilon + e_3)\big) &= 0 &&\text{ in }\;   S^{\varepsilon}\times (0, T],
	\end{aligned}
\end{equation}
where~$e_3 = (0, 0, 1)^\top$.	
In root tissue, water is transported through the xylem, which can be regarded as a porous medium~\citep{hentschel2013simulation}, and hence is also modelled by the Richards equation
\begin{equation}\label{richards_root}
	\partial_t\theta_P(h_{P}^\varepsilon) - \nabla \cdot \big( H^\ve_r K_P(h_{P}^\varepsilon)(\nabla h_{P}^\varepsilon + e_3)\big) = 0 \quad \text{ in } \;\; P^{\varepsilon}\times(0, T].
\end{equation}
In equations \eqref{richards_soil} and \eqref{richards_root}, $K_S^\varepsilon$ and $K_P$ denote the hydraulic conductivity functions of soil and root tissue, and~$\theta_S^\varepsilon$ and $\theta_P$ are the soil and root water retention functions,   and $h^\ve_S$ and $h_P^\ve$ are soil and root xylem pressure heads, respectively.
The hydraulic conductivity of the root tissue is anisotropic, depending in the lateral directions on root circumference and intrinsic radial conductance of the root tissue $k_{\text{r}}$, and intrinsic axial conductance $k_{\text{ax}}$ in the vertical direction:
\begin{equation*}
\begin{aligned}
H_r^\ve = I_\ve H_r, \;  \text{ with } \; 	 H_r =
\begin{pmatrix}
2\pi   r\rho g k_{\text{r}} & 0 & 0\\
0 & 2\pi   r\rho g k_{\text{r}}& 0 \\
0 & 0 & \frac{k_{\text{ax}}\rho g}{L_3}
\end{pmatrix},  \;  \;
I_\varepsilon  =
\begin{pmatrix}
\varepsilon & 0 & 0\\
0 & \varepsilon & 0\\
0 & 0 & 1\\
\end{pmatrix}.
\end{aligned}
\end{equation*}
 Note, only the root radius scales with~$\varepsilon$, not the root length, and,  therefore, only the lateral root hydraulic conductivity values contain a factor $\varepsilon$.

To address the soil heterogeneity, we  consider different hydraulic conductivity and retention functions for  rhizosphere  $\theta_R, K_R$ and bulk soil $\theta_B, K_B$ and define
$$
\begin{aligned}
	&	\theta_S^\varepsilon(x, h_S^\varepsilon) = \theta_S\big( \frac{\hat x}{\varepsilon}, h_S^\varepsilon\big), \qquad K_S^\varepsilon(x, h_S^\varepsilon) = K_S\big( \frac{\hat x}{\varepsilon}, h_S^\varepsilon\big),\\
	& h_S^\varepsilon(t, x) =  h_R^\varepsilon(t, x) \chi_{R^\ve}(x) + h_B^\varepsilon(t, x) \chi_{B^{\varepsilon}}(x),
\end{aligned}
$$
where $\hat x = (x_1, x_2)$ and~$\theta_S(\cdot, h), K_S(\cdot, h)$ are~$Y$-periodic, such that for $y \in Y$ and $h\in\mathbb{R}$
			$$
			\theta_S(y, h)
			= \theta_R(h) \chi_{\hat R}(y)+  \theta_B(h) \chi_{\hat B}(y),
			\;
			 K_S(y, h) = K_R(h) \chi_{\hat R}(y) + K_B(h) \chi_{\hat B}(y).
				$$
Here $\chi_J$ denotes the characteristic function of the subdomain $J= \hat B, \hat R$, extended $Y$-periodically to $\mathbb R^2$, $h_R^\ve$ and $h_B^\ve$ denote  the pressure head in the rhizosphere and  bulk soil, respectively, and $\chi_{R^\ve}(x) = \chi_{\hat R}(\hat x/\ve)$,  $\chi_{B^\ve}(x) = \chi_{\hat B}(\hat x/\ve)$ for $x\in \Omega$.

We consider  the continuity of the pressure head and flux on the interface between rhizosphere and bulk soil
\begin{equation*}
h_{R}^{\varepsilon} = h_B^{\varepsilon}, \quad  K_R(h_{R}^\varepsilon)(\nabla h_{R}^\varepsilon + {e}_3)\cdot\nu = K_B(h_{B}^\varepsilon)(\nabla h_{B}^\varepsilon + e_3)\cdot\nu  \quad \text{ on } \; \Gamma_R^\ve\times(0, T],
\end{equation*}
and across the rhizosphere-root boundary~$\Gamma_P^\ve$, we set the flux of water  proportional to the difference between water pressure head in the soil~$h_S^\varepsilon$ and root tissue~$h_P^\varepsilon$~\citep{zarebanadkouki2016estimation,vanderborght2021hydraulic}.
 Note that in the case of an air-entry pressure head model, see e.g.~\cite{leverett1941capillary}, we would have  a discontinuity in the pressure head at the interface~$\Gamma_R^\varepsilon$ between the rhizosphere and bulk soil.

At the upper boundary defined by the union of the soil surface~$\Gamma_{S, 0}^\varepsilon$ and root collars~$\Gamma_{P, 0}^\varepsilon$, it is assumed that the total water flux, referred to as the crop evapotranspiration~$\text{ET}_\text{c}$, can be formulated as
$$
\text{ET}_\text{c} = (\text{K}_\text{cb}+\text{K}_\text{e})\text{ET}_\text{o}.
$$
Here the constant~$\text{ET}_\text{o}>0$ is the reference evapotranspiration,~$\text{K}_\text{cb}$ is the basal crop coefficient that determines the proportion of~$\text{ET}_\text{c}$ that comes from plant transpiration and~$\text{K}_\text{e}$ is the coefficient that describes the proportion due to evaporation~\citep{allen1998crop}. This setup assumes a constant value for~$\text{K}_\text{cb}>0$ according to average climatic conditions and a function of pressure head~$\text{K}_\text{e}$ that prescribes a rate of evaporation that increases with soil water content. As such, the outward water flux at the root collars~$\Gamma_{P, 0}^\varepsilon$ is equal to the potential transpiration~$\mathcal{T}_{\text{pot}} = \text{K}_\text{cb}\text{ET}_\text{o}$. The function~$f(h_S^\ve)$ for the water flux at the upper soil surface~$\Gamma_{S, 0}^\varepsilon$ is defined as a sum of~$\text{K}_\text{e}(h_S^\varepsilon)\text{ET}_\text{o}$ and functions describing precipitation and surface run-off, see Appendix for further details.
The boundary and initial conditions are therefore stated as
	\begin{equation}
		\begin{aligned}\label{bcs_soil}
			-K_S^\varepsilon(x,h_{S}^\varepsilon)(\nabla h_{S}^\varepsilon + e_3)\cdot\nu &= \varepsilon k_\Gamma (h_S^{\varepsilon}-h_P^\varepsilon) && \text{  on }  \Gamma_P^\ve\times (0, T],\\
			-K_S^\varepsilon(x, h_{S}^\varepsilon)(\nabla h_{S}^\varepsilon + e_3)\cdot\nu &= f(h_{S}^\varepsilon) &&\text{ on } \Gamma_{S, 0}^\varepsilon \times (0, T],\\
			-K_S^\varepsilon(x,h_{S}^\varepsilon)(\nabla h_{S}^\varepsilon + e_3)\cdot\nu &= 0 && \text{ on } \Gamma_N\times (0, T],\\
			h_{S}^\varepsilon &= 0 && \text{ on } \Gamma_{S, L_3}^\varepsilon \times (0, T],\\
			h_{S}^\varepsilon(0) &= h_{S, 0} && \text{ in }  S^\varepsilon,
		\end{aligned}
	\end{equation}
 where  $\Gamma_N = \partial \Omega \setminus (\Gamma_0 \cup \Gamma_{L_3})$, and
		\begin{equation}
		\begin{aligned}\label{bcs_root}
			-H_r^\varepsilon K_P(h_{P}^\varepsilon)(\nabla h_{P}^\varepsilon + {e}_3)\cdot \nu &= \varepsilon k_\Gamma (h_P^{\varepsilon}-h_S^\varepsilon) && \text{ on }  \Gamma_P^\ve\times (0, T],\\
			-H_r^\varepsilon K_P(h_{P}^\varepsilon)(\nabla h_{P}^\varepsilon + {e}_3)\cdot \nu &= \mathcal{T}_{\text{pot}} && \text{ on } \Gamma_{P, 0}^\varepsilon\times (0, T],\\
			h_{P}^\varepsilon &= a && \text{ on } \Gamma_{P, L_3}^\varepsilon\times (0, T],\\
			h_{P}^\varepsilon(0) &= h_{P, 0} && \text{ in }  P^\varepsilon,
		\end{aligned}
	\end{equation}
where~$k_\Gamma =k_{\rm r}\rho g$, with  $\rho$ denoting the water density  and $g$ gravitational acceleration.

\begin{remark}
 Notice that it is possible to consider the $\ve^2$ or $\ve^0$ scaling in  the matrix $I_\ve$  defined above, with the latter resulting into the same macroscopic
model as derived here. In the case where the radial permeability of root branches is
proportional to $\ve^2$,  we would obtain a macroscopic two-scale problem with microscopic
 diffusion across root branches. This may be important for some applications but will
 be computationally more expensive when considering large scale processes.
\end{remark}

\section{Existence and uniqueness results}\label{section:existence_uniqueness}
In this section we prove the well-posedness of  model \eqref{richards_soil}-\eqref{bcs_root}. There are many well-posedness results for degenerate parabolic equations,  but due to specific assumptions on the nonlinear functions and boundary conditions considered here, we present the main steps of the proofs, using the methods as in~\cite{alt1983quasilinear, Jaeger_Kacur_1995}.  We consider the space
$$
 V(J^\ve)=\big\{h\in H^1(J^\ve)\; : \;  h = 0 \; \text{ on } \; \Gamma^\ve_{J,L_3}\big\},
 \qquad \text{ for } \;  J=S,P. $$
For  $u \in V(J^\ve)$ and  $\psi \in V(J^\ve)^\prime$ the dual product is  denoted by   $\langle u, \psi \rangle_{V(J^\ve)^\prime}$
and  the inner product for $\phi \in L^p(A)$ and $\psi \in L^{q}(A)$  by $\la\phi, \psi\ra_{A} = \int_A \phi \psi dx$, given a domain $A \subset \mathbb R^3$ and $1/p+ 1/q =1$. Similar notation is used for the inner product for  $\phi \in L^p(\partial A)$ and $\psi \in L^{q}(\partial A)$. We also denote $\langle u, \psi \rangle_{V(J^\ve)^\prime, \tau} = \int_0^\tau \langle u(t), \psi(t) \rangle_{V(J^\ve)^\prime} dt$ and $A_\tau = (0,\tau) \times A$, where $A$ is a domain or boundary of a domain.

\begin{assumption} \label{assumption}
	\
\begin{itemize}
\item[(A1)] The functions~${\theta}_R, \theta_B, {\theta}_P:\mathbb{R}\to\mathbb{R}$ are strictly increasing and Lipschitz continuous,   $ {\theta}_J( 0) = 0$, and $- \infty<{\theta}_{J,\text{min}}\leq {\theta}_J(z)\leq {\theta}_{J,\text{max}}< + \infty$~for all $z\in\mathbb{R}$, where $J=R,B,P$.
\item[(A2)] The functions~$K_R, K_B, K_P:\mathbb{R}\to[0,\infty)$ are continuous and~$0<K_{J, 0}\leq K_J(z)\leq K_{J,\text{sat}}$, for  $J=R,B, P$ and  for all $z\in\mathbb{R}$.
\item[(A3)] The function~$f: \mathbb R \to \mathbb R$ is continuous and $\lvert f(z)\rvert\leq f_m$ for all~$z\in\mathbb{R}$.
\item[(A4)] The initial conditions  $h_{S,0}\in V(S^\ve)$, $h_{P,0}-a \in V(P^\ve)$ are non-positive.
\end{itemize}
\end{assumption}

\begin{remark}
 The assumption $\theta_J(0)=0$, with $J=R,B,P$, is not restrictive, since we can define shifted functions $\theta_J - \theta_{J,\text{sat}}$  that will satisfy the same problem, where $\theta_{J,\text{sat}}$ denotes  the saturated water content for the soil and root xylem respectively.\\
 Examples of functions satisfying Assumption~\ref{assumption} are given in the Appendix. For the
analysis, it is possible to consider $\mathcal T_{\rm pot}$ as a continuous, bounded function of $h^\ve_P$, with
 the same additional assumptions in the proofs of the uniqueness and non-negativity results as for function $f$.  The assumptions on  $K_J$, for $J=B, R, P$, ensure uniform ellipticity of equations~\eqref{richards_soil} and \eqref{richards_root}. Those assumptions  can be achieved by a regularisation of commonly used  hydraulic conductivity functions, without significant impact on the physically relevant values of the pressure heads, see~the Appendix for an example of a possible regularisation approach.
\end{remark}

\begin{definition}\label{weak solution}
A weak solution of microscopic model~\eqref{richards_soil}--\eqref{bcs_root} is $(h_S^\ve, h_P^\ve)$ with $h_S^\ve\in L^2(0,T;V(S^\ve))$,  $h_P^\ve-a\in L^2(0,T;V(P^\ve))$, $\partial_t {\theta}_S^\ve (\cdot, h_S^\ve)\in L^2(0,T; V(S^\ve)^\prime)$,   $\partial_t {\theta}_P(h_P^\ve)\in L^2(0,T; V(P^\ve)^\prime)$ and  satisfy
\begin{equation}\label{weak_formulation_soil}
\begin{aligned}
\int_0^T\Big[\big\langle\partial_t {\theta}_S^\ve(x, h_S^\ve), \phi \big\rangle_{V(S^\ve)^\prime} + \big\langle K_S^\ve(x, h_{S}^\ve)(\nabla h_{S}^\ve + {e}_3), \nabla \phi \big\rangle_{S^\ve}  \\
+ \ve  k_\Gamma  \big\langle h_S^\ve- h_P^\ve,\phi \big\rangle_{\Gamma_P^\ve} + \big\langle f(h_S^\ve), \phi \big\rangle_{\Gamma_{S,0}^\ve} \Big] dt & =0,
\end{aligned}
\end{equation}
\begin{equation}\label{weak_formulation_plant}
\begin{aligned}
\int_0^T\Big[\big\langle\partial_t {\theta}_P(h_P^\ve), \psi \big\rangle_{V(P^\ve)^\prime}  + \big\langle H_r^\ve K_P(h_{P}^\ve)(\nabla h_{P}^\ve + {e}_3), \nabla \psi \big\rangle_{P^\ve} \\
+ \ve k_\Gamma \big\langle h_P^\ve- h_S^\ve, \psi \big\rangle_{\Gamma_P^\ve} + \big \langle \mathcal{T}_\text{pot}, \psi \big\rangle_{\Gamma^\ve_{P,0}}\Big] dt & =0,
\end{aligned}
\end{equation}
for all $\phi\in L^2(0,T;V(S^\ve))$ and~$\psi\in L^2(0,T;V(P^\ve))$.
\end{definition}

In the proofs  we will use the functions~${\Theta}_P$, $\Theta_R$, ${\Theta}_B$, and $\Theta_S^\ve$, defined as
\begin{equation}\label{big_theta}
\begin{aligned}
{\Theta}_J(h_J)& = {\theta}_J(h_J)h_J + \int_{h_J}^0 {\theta}_J(z)dz \geq 0, \quad \text{ for } \; J=P, R,B, \\
{\Theta}^\ve_S(x, h_S) &=  {\Theta}_R(h_S) \chi_{R^\ve}(x) +  {\Theta}_B(h_S) \chi_{B^\ve}(x).
\end{aligned}
\end{equation}
The definition of $\Theta_J$ and monotonicity assumptions on $\theta_J$, for $J= P, R, B$,   imply
\begin{equation}\label{rslt_bg_sml_theta}
\begin{aligned}
{\Theta}_J(u) -  {\Theta}_J(v) =  {\theta}_J(u)u + \int_{u}^0{\theta}_J(z)dz - \Big[{\theta}_J(v)  v + \int_{v}^0{\theta}_J(z)dz\Big]
\\  \leq\big( {\theta}_J(u) - {\theta}_J(v)\big) u.
\end{aligned}
\end{equation}

\begin{remark}
For simplicity of presentation, in the proofs of existence, uniqueness, and non-positivity results, obtained for each fixed $\ve>0$,  we shall use the notation $\theta_S(u):= \theta_S^\ve(x,u)$,  $\Theta_S(u) :=\Theta_S^\ve(x, u)$, and  $K_S(u):=K^\ve_S(x, u)$.   Also, for the sake of brevity, in the proofs below we will neither specify nor relabel subsequences for which the convergence results hold.
\end{remark}

\begin{remark}
 Notice that for each fixed $\ve>0$ the proof of existence of a solution to model~~\eqref{richards_soil}--\eqref{bcs_root} follows the same ideas and steps as in~\cite{alt1983quasilinear} and in wider meaning can be inferred from the  results in~\cite{alt1983quasilinear}, obtained for more general systems of quasilinear elliptic-parabolic equations.  However, due to specific boundary conditions at the root-soil interface, not considered in~\cite{alt1983quasilinear},  we include the main steps of the proof for completeness.
\end{remark}

\begin{theorem}\label{theorem_microscopic_models_existence}
Under Assumption~\ref{assumption} for each fixed $\ve >0$ there exists a weak solution~$(h_S^\ve, h_P^\ve)$ to model~\eqref{richards_soil}-\eqref{bcs_root}.
\end{theorem}

\begin{proof}
The existence of a solution can be  shown using the Rothe-Galerkin method,
see e.g.~\citep{alt1983quasilinear, kavcur2006method}. Since $\ve >0$ is fixed, for the simplicity of notation we  omit the superscript~$\ve$ in $h^\ve_S$ and $h^\ve_P$. With $i=1, \ldots, N$, where~$N = T/ \tau \in \mathbb N$ for some $\tau >0$, and~$m\in\mathbb{N}$, we consider
$$
h_{S,i}^m = \sum_{k=1}^m\alpha_{i,k}^mv_k \quad \text{ and } \quad h_{P,i}^m - a = \sum_{k=1}^m\beta_{i,k}^mw_k,
$$
where $	\{v_k\}_{k=1}^\infty$ and $\{w_k\}_{k=1}^\infty$ are orthogonal bases for $V(S^\ve)$ and  $V(P^\ve)$,  orthonormal in $L^2(S^\ve)$ and $L^2(P^\ve)$, respectively, and
the  discrete-in-time equations
\begin{equation}\label{discrete_space_time_soil}
	\begin{aligned}
		\frac{1}{\tau}\big\langle  \theta_S(h_{S,i}^m) - \theta_S(h_{S,i-1}^m), \phi\big\rangle_{S^\ve} + \big\langle K_S(h_{S,i-1}^m)(\nabla h_{S,i}^m + {e}_3), \nabla \phi \big\rangle_{S^\ve}\\
		+  \ve\, k_\Gamma \big\langle h_{S,i}^m - h_{P,i-1}^m, \phi \big\rangle_{\Gamma_P^\ve}
		+  \big\langle f(h_{S, i-1}^m), \phi \big\rangle_{\Gamma_{S,0}^\ve}  = 0,
	\end{aligned}
\end{equation}
\begin{equation}\label{discrete_space_time_root}
	\begin{aligned}
		\frac{1}{\tau}\big\langle  \theta_P (h_{P,i}^m) - {\theta}_{P}(h_{P,i-1}^m), \psi \big\rangle_{P^\ve}
		+ \big\langle H_r^\ve K_P(h_{P,i-1}^m)(\nabla h_{P,i}^m + {e}_3), \nabla \psi \big\rangle_{P^\ve}\\
		-  \ve  \, k_\Gamma \langle h_{S,i-1}^m - h_{P,i}^m, \psi \rangle_{\Gamma_P^\ve}
		+ \langle \mathcal{T}_\text{pot}, \psi \rangle_{\Gamma_{P,0}^\ve} = 0,
	\end{aligned}
\end{equation}
for all~$\phi\in V_m$,~$\psi\in U_m$, with $V_m= {\rm span}\{v_1, ..., v_m\} $ and~$U_m = {\rm span} \{w_1, ..., w_m\}$ and  $h_{S, 0}^m$ and $h_{P, 0}^m-a$ being the projections of~$h_{S,0}$ and $h_{P,0}-a$ onto~$V_m$ and~$U_m$ respectively. The regularity of~$h_{S,0}$ and~$h_{P,0}$ implies~$h_{S, 0}^m\to h_{S,0}$ in~$V(S^\ve)$ and $h_{P, 0}^m-a\to h_{P,0}-a$ in~$V(P^\ve)$ as~$m\to\infty$.

Similar arguments as in~\cite{alt1983quasilinear, Ptashnyk_2006}, imply that   for given $(h^m_{S,i-1}, h^m_{P, i-1})$, under Assumption~\ref{assumption}, there exists solution~$(h_{S,i}^m,h_{P,i}^m)$ to~\eqref{discrete_space_time_soil} and~\eqref{discrete_space_time_root}, for all~$i=1, \ldots, N$.
To derive a priori estimates
we consider $h_{S,i}^m$ and~$h_{P,i}^m -a$ as  test functions in~\eqref{discrete_space_time_soil} and~\eqref{discrete_space_time_root}. Summing over $i=1, \ldots, n$, for  $1\leq n\leq N$, multiplying by $\tau$, and  using Assumption~\ref{assumption}, inequality~\eqref{rslt_bg_sml_theta}, and properties of the telescoping sum, yields
\begin{equation*}
\begin{aligned}
 \sum_{J=S,P}\! \Big[\big\| {\Theta}_J\big(h_{J,n}^m\big)\big \|_{L^1(J^\ve)}
\! +  \frac{\kappa_{J,0}}2 \!\sum_{i=1}^n \! \tau  \|\nabla h_{J,i}^m\|^2_{L^2(J^\ve)} \!  +     \frac{k_\Gamma}2    \tau \ve  \big\| h_{J,n}^m\big\|^2_{L^2(\Gamma_P^\ve)} \Big]
\\\leq  C_\delta
 + \! \sum_{J=S,P}\! \Big[ \sum_{i=1}^n \! \tau \delta \big(  \big\| h_{J,i}^m\big\|^2_{L^2(\Gamma_{J,0}^\ve)}\! +  \ve  \big\| h_{J,i}^m\big\|^2_{L^2(\Gamma_{P}^\ve)} \big) \qquad \qquad \\
+ \big \|\Theta_J\big(h_{J,0}^m\big) \big \|_{L^1(J^\ve)} \!  +   \big(\frac{k_\Gamma}2 + \delta\big)    \tau \ve \big\| h_{J,0}^m\big\|^2_{L^2(\Gamma_P^\ve)} \Big],
\end{aligned}
\end{equation*}
for some $\delta >0$,
where  $\kappa_{S,0}=\min\{K_{R,0}, K_{B,0}\}$,  $\kappa_{P,0}= \rho g  K_{P,0} \min\{2\pi \ve r k_{\rm r}, k_{\rm ax}/ L_3 \}$,
 and  the Cauchy-Schwarz inequality and  properties of the telescoping sum imply
\begin{equation*}
\begin{aligned}
&\sum_{i=1}^n\tau \int_{\Gamma_{P}^\ve}\! \!\Big(\big\lvert h_{S,i}^m\big\rvert^2 -h_{P,i-1}^mh_{S,i}^m  -h_{S,i-1}^mh_{P,i}^m +\big\lvert h_{P,i}^m\big\rvert^2\Big)d\gamma\\
& \qquad \geq \frac{\tau }{2}\Big(\big\| h_{S,n}^m\big\|^2_{L^2(\Gamma_P^\ve)}  +\big\| h_{P,n}^m\big\|^2_{L^2(\Gamma_P^\ve)}-\big\| h_{S,0}^m\big\|^2_{L^2(\Gamma_P^\ve)} - \big\| h_{P,0}^m\big\|^2_{L^2(\Gamma_P^\ve)}\Big).
\end{aligned}
\end{equation*}
To estimate the boundary integrals, we use the trace theorem, the generalised Poincar\'e inequality, and~$h_{S,i}^m=0$ on $\Gamma_{S, L_3}$ and~$h_{P,i}^m=a$ on $\Gamma_{P, L_3}$, to obtain
\begin{equation}\label{trace_poincare_dirichlet_1}
\sum_{i=1}^n\tau \Big(\ve \| h_{J,i}^m\|^2_{L^2(\Gamma_P^\ve)} +  \| h_{J,i}^m\|^2_{L^2(\Gamma_{J,0}^\ve)} \Big)
\leq C_1\sum_{i=1}^n\tau \|\nabla h_{J,i}^m\|^2_{L^2(J^\ve)} + C_2a^2,
\end{equation}
for $J=S,P$, where  $C_1, C_2>0$ may depend on $\ve$ and $C_2 = 0$ when $J=S$.
Combining the estimates from above and using assumptions on initial conditions, together with the Poincar\'e inequality, yield
\begin{equation}\label{apriori_1}
 \max\limits_{1 \leq i \leq N} \|\Theta_J(h_{J,i}^m) \|_{L^1(J^\ve)} \leq C, \quad \sum_{i=1}^N\tau \Big[\| h_{J,i}^m\|^2_{L^2(J^\ve)} +  \|\nabla h_{J,i}^m\|^2_{L^2(J^\ve)}\Big] \leq C,
\end{equation}
for $J=S,P$, where $C>0$ may depend on $\varepsilon$.
Summing~\eqref{discrete_space_time_soil} over~$j = i+1, ..., i+l$, with~$l = 1,...,N-1$, and using the properties of the telescoping sum, then taking
$\phi = h_{S, i+l}^m - h_{S, i}^m$ as a test function, and summing over~$i = 1,...,N-l$, gives
$$
\begin{aligned}
&\sum_{i=1}^{N-l}\frac{1}{\tau}\big\langle \theta_{S}(h_{S,i+l}^m)- \theta_{S}(h_{S,i}^m), h_{S, i+l}^m - h_{S, i}^m\big \rangle_{S^\ve}\\
&\leq \sum_{i=1}^{N-l}\sum_{j=i+1}^{i+l} \Big|\big \langle K_S(h_{S,j-1}^m)(\nabla h_{S,j}^m + {e}_3), \nabla(h_{S, i+l}^m - h_{S, i}^m) \big \rangle_{S^\ve}\Big|\\
&\quad +\sum_{i=1}^{N-l}\sum_{j=i+1}^{i+l}\Big|\ve k_\Gamma \big\langle  h_{S,j}^m-h_{P,j-1}^m, h_{S, i+l}^m - h_{S, i}^m\big\rangle_{\Gamma_P^\ve} + \big\langle f(h_{S, j-1}^m), h_{S, i+l}^m - h_{S, i}^m\big\rangle_{\Gamma_{S,0}^\ve}\Big|.
\end{aligned}
$$
Performing similar calculations with
$\phi = h_{P, i+l}^m - h_{P, i}^m$ as test function in \eqref{discrete_space_time_root} and using  estimates on $h_{S,j}^m$ and $h_{P,j}^m$ in \eqref{apriori_1} imply
\begin{equation}\label{estim_diff_t_S}
\sum_{i=1}^{N-l}\tau\big\langle {\theta}_{J}(h_{J,i+l}^m)- {\theta}_{J}(h_{J,i}^m), h_{J, i+l}^m - h_{J, i}^m\big \rangle_{J^\ve} \leq C l \tau, \qquad \text{ for } \; J = S,P.
\end{equation}
To pass to the limit in~\eqref{discrete_space_time_soil} and~\eqref{discrete_space_time_root} as $m, N \to \infty$
we consider the piecewise constant interpolations in time
\begin{equation}\label{piecewise_interploations}
\overline{h}_{J,N}^m(t,x) = h_{J,i}^m(x) \; \text{ for }\; t\in(\tau(i-1),~\tau i] \text{ with }  i=1,...,N,\;  \text{ and } x \in J^\ve,
\end{equation}
and $\hat{h}_{J,N}^m(t, x) = \overline{h}_{J,N}^m(t-\tau,x) = h_{J,i-1}^m(x)$, where  $\hat{h}_{J,N}^m(t,x) = h_{J,0}^m(x)$  for~$t\in(0,\tau]$, for~$J=S,P$.
Then   \eqref{apriori_1} and  \eqref{estim_diff_t_S} yield
\begin{eqnarray}\label{eqtn_continuous_analogues}
\Big\langle {\theta}_{J}(\overline{h}_{J,N}^m(\cdot+\lambda))- {\theta}_{J}(\overline{h}_{J,N}^m), \, \overline{h}_{J,N}^m(\cdot+\lambda) - \overline{h}_{J,N}^m\Big\rangle_{(0,T-\lambda)\times J^\ve} \leq  C\lambda, \\
 \sup_{0\leq t\leq T}\big\| {\Theta}_J(\overline{h}_{J,N}^m(t))\big\|_{L^1(J^\ve)}  + \big\|\overline{h}_{J,N}^m\big\|_{L^2(0,T; H^1(J^\ve))}   + \big\|\hat{h}_{J,N}^m\big\|_{L^2(0,T; H^1(J^\ve))}   \leq C, \nonumber
\end{eqnarray}
for $J=S,P$, uniform with respect to~$m$ and~$N$, where~$\lambda\in(l\tau, (l+1)\tau)$ for some~$l=$~$0,...,N-1$.
Using estimates \eqref{eqtn_continuous_analogues} and equations  \eqref{discrete_space_time_soil} and \eqref{discrete_space_time_root} we obtain
\begin{equation}\label{equation_lemma_discrete_time_derivative}
\big\|\partial_{t}^\tau {\theta}_{S}\big(\overline{h}_{S,N}^m\big)\big\|_{L^2(0,T;V(S^\ve)^\prime)} +
\big\|\partial_{t}^\tau {\theta}_{P}\big(\overline{h}_{P,N}^m\big)\big\|_{L^2(0,T;V(P^\ve)^\prime)} \leq C,
\end{equation}
where
$$
\partial_{t}^\tau {\theta}_{J} (\overline{h}_{J,N}^m)= \frac{1}{\tau}\Big[ {\theta}_{J}(\overline{h}_{J,N}^m)- {\theta}_{J}(\hat{h}_{J,N}^m)\Big], \qquad \text{ for } J = S,P.
$$
Combining estimates \eqref{eqtn_continuous_analogues} and \eqref{equation_lemma_discrete_time_derivative} yields
\begin{equation}\label{weak_convergence_h}
\begin{aligned}
\hat{h}_{J,N}^m\rightharpoonup \hat{h}_J, \quad 		\overline{h}_{J,N}^m\rightharpoonup h_J &\;  \text{ weakly in }L^2(0,T; H^1(J^\ve)),  &&  J=S,P, \\
\partial_{t}^\tau {\theta}_{J}(\overline{h}_{J,N}^m)\rightharpoonup \mu_J  &  \; \text{ weakly in } L^2(0,T; V(J^\ve)^\prime),
 && \text{ as } m,N\to \infty.
\end{aligned}
\end{equation}
The boundedness, Lipschitz continuity and monotonicity  of~$ {\theta}_J$  imply
\begin{equation}\label{eqn_kolmogorov_compactness}
\begin{aligned}
&\big\| {\theta}_J(\overline{h}_{J,N}^m)\big\|_{L^2((0,T)\times J^\ve)} \leq C, \\
& \big\| \nabla {\theta}_J \big(\overline{h}_{J,N}^m\big) \big \|_{L^2((0,T)\times J^\ve)} \leq L_{\theta_J}
\|\nabla \overline{h}_{J,N}^m \|_{L^2((0,T)\times J^\ve)}, \\
&\big\| {\theta}_J \big(\overline{h}_{J,N}^m(\cdot + \lambda)\big)- {\theta}_J\big(\overline{h}_{J,N}^m\big)\big\|_{L^2((0,T-\lambda)\times J^\ve)}^2\\
&  \quad \leq L_{\theta_J}\big\langle \theta_J\big(\overline{h}_{J,N}^m(\cdot+\lambda)\big)- {\theta}_J\big(\overline{h}_{J,N}^m\big), \overline{h}_{J,N}^m(\cdot+\lambda)-\overline{h}_{J,N}^m\big\rangle_{J^\ve_{T-\lambda}} \leq  C \lambda,
\end{aligned}
\end{equation}
for $J=P,R,B$, where $C>0$ is independent of~$m$ and~$N$, and   $L_{\theta_J}$ are the Lipschitz constants for $\theta_J$.
Using~\citep[Theorem 1]{simon1986compact}, there exists $z_J \in  ~L^2((0,T)\times J^\ve)$,  for $J=P,R,B$,  such that
\begin{equation}\label{eqn_result_theta_strong_L2_space_time_convergence}
\begin{aligned}
{\theta}_J\big(\overline{h}_{J,N}^m\big)\to z_J\qquad\text{strongly in}~L^2((0,T)\times J^\ve) \quad \text{ as } \; m,N\to\infty.
\end{aligned}
\end{equation}
Since~$ {\theta}_J$ is strictly increasing and continuous, it admits a continuous inverse and
\begin{equation}\label{eqn_result_ae_convergence_h_to_?}
\overline{h}_{J,N}^m =  {\theta}_J^{-1}\big( {\theta}_J(\overline{h}_{J,N}^m)\big)\to {\theta}_J^{-1}(z_J) \quad \text{ a.e.~in } \; (0,T)\times J^\ve \quad \text{ as } \; m,N\to\infty,
\end{equation}
and,  because~$\overline{h}_{J,N}^m\rightharpoonup h_J$  in~$L^2(0,T; H^1(J^\ve))$ as~$m,N\to\infty$, we have $ {\theta}_J^{-1}\big(z_J\big) = h_J $ and~$z_J =  {\theta}_J\big(h_J\big)$, for $J=P,R,B$.
Applying the change of variables~$t=t+\lambda$, with~$\lambda = \tau$, in the last estimate in \eqref{eqn_kolmogorov_compactness} and using the monotonicity and Lipschitz continuity of $\theta_J$  yield
\begin{eqnarray}\label{eqn8_lemma_theta_strong_hat}
 \big\| {\theta}_{J}\big(\overline{h}_{J,N}^m\big) -  {\theta}_{J}\big(\hat{h}_{J,N}^m\big)\big\|_{L^2((0,T)\times J^\ve)}^2  \leq
L_{\theta_J}\Big[  \tau \big\langle {\theta}_{J}(h_{J,1}^m)- {\theta}_{J}(h_{J,0}^m), h_{J,1}^m-h_{J,0}^m\big\rangle_{J^\ve} \nonumber
\\
 + \int_\tau^T \!\!\!\langle {\theta}_{J}\big(\overline{h}_{J,N}^m\big)- {\theta}_{J}\big(\hat{h}_{J,N}^m\big), \overline{h}_{J,N}^m - \hat{h}_{J,N}^m\rangle_{J^\ve}dt \Big]
\leq   C\tau^{\frac 12},
\end{eqnarray}
 for $J=P,R,B$. Thus, using \eqref{eqn_result_ae_convergence_h_to_?}, we have
		\begin{equation}\label{eqn9_lemma_theta_strong_hat}
				 {\theta}_{J}\big(\hat{h}_{J,N}^m\big)\to {\theta}_{J}\big(h_{J}\big)\quad\text{strongly in }L^2((0,T)\times J^\ve) \, \text{ as } \,  m, N\to\infty,
		\end{equation}
and  $\hat{h}_{J,N}^m = \theta_J^{-1}(\theta_J(\hat{h}_{J,N}^m)) \to \theta_J^{-1}(\theta_J(h_J)) =h_J$ a.e.~in $(0,T)\times J^\ve$  as $m,N\to\infty$, and
\begin{equation}\label{eqn11_lemma_h_hat_weakly_h}
				\hat{h}_{J,N}^m  \rightharpoonup h_J \quad\text{weakly in } \; L^2\big(0,T; H^1(J^\ve)\big) \; \text{ as } \; m,N\to\infty, \; \text{ for } \;  J=P,S.
\end{equation}
 Boundedness of $\theta_J$ ensures $\partial_{t}^\tau {\theta}_{J}(\overline{h}_{J,N}^m)\in L^2((0,T)\times J^\ve)$ and
\begin{equation}\label{eqn_convergence_diff_quotient_to_time_derivative}
\begin{aligned}
\int_0^T\!\!\!\big\langle\partial_{t}^\tau\big( {\theta}_{J}\big(\overline{h}_{J,N}^m\big)\big), w\big\rangle_{V(J^\ve)^\prime} dt
	\to 
	- \int_0^T\!\!\!\big\langle {\theta}_{J}\big(h_{J}\big), \partial_{t}w\big\rangle_{V(J^\ve)^\prime} dt, \; \text{ as } m,N\to\infty,
			\end{aligned}
		\end{equation}
	for $w \in C_0^\infty(0,T; C_{\Gamma_{J, L_3}}^\infty(\overline{J^\ve}))$,   where  $C_{\Gamma_{J,L_3}}^\infty(\overline{J^\ve}) = \{v\in C^\infty(\overline{J^\ve})\, : \, v=0 \text{ on } \Gamma^\ve_{J,L_3}\}$, and  $J=P,S$. Considering $w = \kappa\varphi$  with  $\kappa\in C^\infty_0(0,T)$ and~$\varphi\in C_{\Gamma_{J, L_3}}^\infty(\overline{J^\ve})$ and using the convergence results in~\eqref{weak_convergence_h}, together with the definition of the weak derivative,  yields $ \mu_J = \partial_t {\theta}_J\big(h_{J}\big)$ for $J=P, S$.
Continuity of $K_J$ implies  $K_J\big(\hat{h}_{J,N}^m\big)\to K_J(h_J)$  a.e.~in $(0,T)\times J^\ve$  and, since~$|K_J(\hat{h}_{J,N}^m)|\leq K_{\rm max}$, for some $K_{\rm max}>0$,   by the dominated convergence theorem we have
		\begin{equation}\label{eqn_Kn_to_K_strongly}
		\begin{aligned}
				& K_J(\hat{h}_{J,N}^m)\to K_J(h_J) &&  \text{strongly in }  L^p((0,T)\times J^\ve),   \\
				& K_J(\hat{h}_{J,N}^m)(\nabla\overline{h}_{J,N}^m + {e}_3) \rightharpoonup K_J(h_J)(\nabla h_J + {e}_3)   && \text{weakly in }  L^q((0,T)\times J^\ve)^3,
				\end{aligned}
		\end{equation}
	as~$m,N\to\infty$, with $p>2$ and $q = 2p/(p+2)$, and $J =P,S$.
	Using the trace inequality applied to $|{\theta}_{S}\big(\hat{h}_{S,N}^m\big) -  {\theta}_{S}\big(h_{S}\big)|^2$, along with the Lipschitz continuity of $\theta_S$ and boundedness of $\hat{h}_{S,N}^m$ in $L^2(0,T; H^1(S^\ve))$,  yields
$$
\begin{aligned}
\big\| {\theta}_{S}\big(\hat{h}_{S,N}^m\big) -  {\theta}_{S}\big(h_{S}\big)\big\|_{L^2((0,T)\times \Gamma^\ve_{S,0})}^2 \leq
C \Big[\big\| {\theta}_{S}\big(\hat{h}_{S,N}^m\big) -  {\theta}_{S}\big(h_{S}\big)\big\|_{L^2( S^\ve_T)}^2 \\
+ \| {\theta}_{S}\big(\hat{h}_{S,N}^m\big) -  {\theta}_{S}\big(h_{S}\big)\big\|_{L^2(S^\ve_T)} \big( \|\nabla \hat{h}_{S,N}^m\|_{L^2(S^\ve_T)} +\|\nabla h_{S}\|_{L^2(S^\ve_T)}\big)  \Big].
\end{aligned}
$$
Combining this with  the convergence of~$ {\theta}_{S}(\hat{h}_{S,N}^m) $ in~$L^2((0,T)\times S^\ve)$ implies
\begin{equation}\label{eqn_strong_convergence_theta_hat_gamma}
{\theta}_{S}(\hat{h}_{S,N}^m)\to  {\theta}_{S}(h_{S})\qquad\text{strongly in }L^2((0,T)\times \Gamma^\ve_{S,0}),
\end{equation}
and $\hat{h}_{S,N}^m = {\theta}_S^{-1}( {\theta}_S(\hat{h}_{S,N}^m)) \to {\theta}_S^{-1}( {\theta}_{S}(h_{S})) = h_S$ a.e.~in $(0,T)\times\Gamma^\ve_{S,0}$,  as  $m,N\to\infty$.
Then the continuity and boundedness of~$f$ ensure the convergence of $ f(\hat{h}_{S,N}^m)$.

Taking~$\phi\in C_0^\infty(0,T; C_{\Gamma_{S,L_3}}^\infty(\overline{S^\ve}))$  and $\psi\in C_0^\infty(0,T; C_{\Gamma_{P,L_3}}^\infty(\overline{P^\ve}))$	as  test functions in~\eqref{discrete_space_time_soil} and~\eqref{discrete_space_time_root}, respectively,   and considering the limits as $m,N \to \infty$, we obtain that $(h_S, h_P)$ is a weak solution of~\eqref{richards_soil}-\eqref{bcs_root}.
\end{proof}

\begin{remark}
In some models  no water flux  at root tips has been considered, i.e.~$ -H_r^\ve K_P(h_{P})(\nabla h_{P} + {e}_3)\cdot \nu   = 0$   on  $(0,T]\times  \Gamma^\ve_{P,L_3}$. Theorem~\ref{theorem_microscopic_models_existence} holds also for such boundary conditions and the only difference  is in the estimation of the boundary integral.
From~\eqref{discrete_space_time_root}, applying the trace theorem and  Poincar\'e inequality,  we obtain
\begin{equation*}\label{eq8_dscrt_es}
\begin{aligned}
&\big\| {\Theta}_P\big(h_{P,n}^m\big)\big\|_{L^1(P^\ve)} + \sum_{i=1}^n \tau \Big[\|\nabla h_{P,i}^m\|^2_{L^2(P^\ve)} + \ve\| h_{P,i}^m\|^2_{L^2(\Gamma_P^\ve)} \Big] \\
& \qquad \leq C\Big[ 1 +   \sum_{i=1}^n\tau \ve \| h_{S,i-1}^m\|^2_{L^2(\Gamma_P^\ve)} + \big\| {\Theta}_{P}\big(h_{P,0}^m\big)\big\|_{L^1(P^\ve)}\Big].
\end{aligned}
\end{equation*}
Thus in the equation for $h_{S,i}^m$ the boundary integral involving $h_{P,i}^m$ can be estimated
$$
\sum_{i=1}^n\tau \ve \| h_{P,i}^m\|^2_{L^2(\Gamma_P^\ve)}\leq
C\Big[ 1 +  \sum_{i=1}^n\tau \ve \| h_{S,i-1}^m\|^2_{L^2(\Gamma_P^\ve)} \Big]
\leq C \Big[1+ \sum_{i=1}^n\tau  \| \nabla h_{S,i}^m\|^2_{L^2(S^\ve)} \Big].
$$
\end{remark}

\begin{theorem}\label{theorem_microscopic_models_uniqueness}
Under  Assumption~\ref{assumption} and, additionally, if $f$ is non-decreasing and~$K_R$, $K_B$, and $K_P$ are Lipschitz continuous, the weak solution~$(h_S^\ve, h_P^\ve)$ of model \eqref{richards_soil}-\eqref{bcs_root} is unique.
\end{theorem}

The proof of Theorem~\ref{theorem_microscopic_models_uniqueness} employs the same method as  in~\citep{otto1996l1}. For this we first prove two inequalities.  Consider  $\sigma_{\delta}^+, \sigma_{\delta}^-:\mathbb{R}\to[0,1]$ and $\eta_{J, \delta}^+$,~$\eta_{J,\delta}^-:\mathbb{R}\to [0,\infty)$, for $\delta >0$,   given by
\begin{equation}\label{our_eta_prime}
\begin{aligned}
\sigma_{\delta}^+(z) = \begin{cases}
1 &\text{for }z\geq\delta,\\
\frac{z}{\delta} &\text{for }0<z<\delta,\\
0 &\text{for }z\leq0,
\end{cases}  \qquad \quad
\begin{aligned} &\sigma_{\delta}^-(z)=-\sigma_{\delta}^+(-z), \\
&\eta_{J,\delta}^\pm\big(h,v^0\big) = \int_{v^0}^h\sigma_\delta^\pm(z-v^0) {\theta}_J^{'}(z)dz,
\end{aligned}
\end{aligned}
\end{equation}
for $J=P,R,B$, where~$v^0\in V(S^\ve)$ if $J=R, B$ and~$v^0-a\in V(P^\ve)$ if $J=P$, and $\eta_{S,\delta}^\pm\big(h,v^0\big) = \eta_{R,\delta}^\pm\big(h,v^0\big)\chi_{R^\ve} + \eta_{B,\delta}^\pm\big(h,v^0\big)\chi_{B^\ve}$. Then for~$\kappa\in C^\infty_0(-\infty,T)$ and  for $v^0\in V(S^\ve)$ and~$w^0-a\in V(P^\ve)$ we have $\sigma_\delta^\pm(h_S-v^0)\kappa \in L^2(0,T;V(S^\ve))$ and $\sigma_\delta^\pm(h_P-w^0)\kappa\in L^2(0,T;V(P^\ve))$.

Similar as in the proof of Theorem~\ref{theorem_microscopic_models_existence}, we  omit the superscript $\ve$ in $h_S^\ve$ and $h^\ve_P$.

\begin{lemma}\label{lemma_uniqueness_lemma_1}
Let~$(h_S,h_P)$ be a solution to~\eqref{richards_soil}-\eqref{bcs_root}  and let $\kappa\in C^\infty_0(-\infty,T)$ be non-negative. Then  for  $v^0\in V(S^\ve)$, $w^0-a\in V(P^\ve)$,  and any~$\delta>0$,  we have
\begin{eqnarray}\label{inequality_uniqueness_lemma_1}
\int_0^T\!\! \Big[\Big\langle \eta_{S,\delta}^+\big(h_{S,0},v^0\big) -\eta_{S,\delta}^+(h_{S},v^0), \frac{d\kappa}{dt}\Big\rangle_{S^\ve}+
\Big\langle \eta_{P,\delta}^+\big(h_{P,0},w^0\big)-\eta_{P,\delta}^+(h_{P},w^0), \frac{d\kappa}{dt}\Big\rangle_{P^\ve} \nonumber\\
+ \Big(\big\langle K_S(h_S)(\nabla h_{S} + {e}_3),\! \nabla \sigma_\delta^+(h_S-v^0) \big \rangle_{S^\ve}\!\! + \big \langle H_r^\ve K_P(h_P)(\nabla h_{P} + {e}_3), \! \nabla \sigma_\delta^+(h_P-w^0)\big \rangle_{P^\ve} \nonumber \\
+ \ve   k_\Gamma  \big\langle h_S-h_P, \sigma_\delta^+(h_S-v^0) - \sigma_\delta^+(h_P-w^0) \big \rangle_{\Gamma_P^\ve} \nonumber  \\
+ \big\langle f(h_S), \sigma_\delta^+(h_S-v^0) \big\rangle_{\Gamma_{S,0}^\ve}
+ \big \langle \mathcal{T}_\text{pot},\sigma_\delta^+(h_P-w^0) \big\rangle_{\Gamma_{P,0}^\ve }\Big) \kappa (t) \Big]dt \leq 0.
\end{eqnarray}
The same inequality holds  with $\eta_{S,\delta}^-(h_{S},v^0)$ and  $\eta_{P,\delta}^-(h_{P},w^0)$ instead of $\eta_{S,\delta}^+(h_{S},v^0)$ and  $\eta_{P,\delta}^+(h_{P},w^0)$,  and   $\sigma_\delta^-(h_S-v^0)$ and  $\sigma_\delta^-(h_P-w^0)$ 	instead of $\sigma_\delta^+(h_S-v^0)$ and  $\sigma_\delta^+(h_P-w^0)$.
\end{lemma}
\begin{proof}
We prove~\eqref{inequality_uniqueness_lemma_1} and the proof for $\eta_{S,\delta}^-(h_{S},v^0)$, $\eta_{P,\delta}^-(h_{P},w^0)$,  $\sigma_\delta^-(h_S-v^0)$,  $\sigma_\delta^-(h_P-w^0)$ follows the same lines. Considering
\begin{equation}\label{def:zeta}
\zeta_\tau(t,x) = \frac{1}{\tau}\int_{t}^{t+\tau}\!\!\zeta(s,x)ds, \; \text{ where }  \;\zeta = \sigma_\delta^+ (h_S-v^0)\kappa\in L^2(0,T;V(S^\ve)),
\end{equation}
for~$t\in[0,T]$ and $\tau>0$,  and using $\kappa(T)=0$  we can write
\begin{equation}\label{eqn_ibp_argument}
\begin{aligned}
	\int_0^T\!\!\big\langle\partial_t {\theta}_S(h_{S}),\zeta_\tau\big\rangle_{V^\prime(S^\ve)} dt
				= \int_0^{T-\tau}\!\! \frac{1}{\tau}\big\langle {\theta}_S(h_{S,0})- {\theta}_S(h_{S}(t)),\zeta(t+\tau)\big\rangle_{S^\ve}dt
\\  - \int_0^T\!\!\frac{1}{\tau} \big\langle{\theta}_S(h_{S,0})- {\theta}_S(h_{S}(t)), \zeta(t)\big\rangle_{S^\ve} dt
\\  = \int_0^T\!\!\frac{1}{\tau}\big \langle {\theta}_S(h_S(t))- {\theta}_S(h_S(t-\tau)),\zeta(t) \rangle_{S^\ve}dt,
\end{aligned}
\end{equation}
where $ {\theta}_S(h_S)$ is extended  by $ {\theta}_S(h_S) =  {\theta}_S(h_{S,0})$ for $t<0$.
Since~$\theta_R, \theta_B$ are Lipschitz continuous and~$\theta_R, \theta_B$ and~$\sigma_\delta^+$ are monotone increasing, it follows that
		\begin{equation}\label{eqn3_uniqueness_lemma_1}
			\begin{aligned}
				\eta_{S,\delta}^+\big(h_S(t),v^0\big) - \eta_{S,\delta}^+\big(h_S(t-\tau),v^0\big)
				= \int_{h_S(t-\tau)}^{h_S(t)}\sigma_\delta^+(z-v^0) {\theta}_S^{'}(z)dz\\
				\leq
				\sigma_\delta^+\big(h_S(t)-v^0\big)\big[ {\theta}_S(h_S(t))- {\theta}_S(h_S(t-\tau))\big].
			\end{aligned}
		\end{equation}
Using \eqref{eqn3_uniqueness_lemma_1}  in  \eqref{eqn_ibp_argument}, together with the non-negativity of~$\kappa\in C^\infty_0(-\infty,T)$, yields
\begin{equation}\label{eqn_time_derivative_lower_bound_uniq_lemma_1}
\begin{aligned}
&\liminf_{\tau\to 0}\int_0^T\!\!\!\big\langle\partial_t {\theta}_S(h_{S}),\zeta_\tau\big\rangle_{V^\prime(S^\ve)} dt
\\
&\geq \liminf_{\tau\to 0} \int_0^T \frac{1}{\tau} \big\langle\eta_{S,\delta}^+(h_S(t),v^0) - \eta_{S,\delta}^+(h_S(t-\tau),v^0), \kappa(t) \big \rangle_{S^\ve} dt
\\
&= \liminf_{\tau\to 0}	\int_0^T\big\langle \eta_{S,\delta}^+(h_{S,0},v^0) - \eta_{S,\delta}^+(h_S(t),v^0), \frac{1}{\tau}\big(\kappa(t+\tau)-\kappa(t)\big) \big \rangle_{S^\ve}dt \\
& = \int_0^T \!\!\!\Big \langle\eta_{S,\delta}^+(h_{S,0},v^0) - \eta_{S,\delta}^+(h_S(t),v^0), \frac{d \kappa}{dt}\Big \rangle_{S^\ve} dt.
\end{aligned}
\end{equation}
The same calculations hold for $\partial_t \theta_P(h_P)$ and $\eta_{P,\delta}^+(h_P, w^0)$, considering $\zeta = \sigma_\delta^+ (h_P-w^0) \kappa$ in the definition of $\zeta_\tau$ in \eqref{def:zeta}. Using the Cauchy-Schwarz inequality,  change in the order of integration and that $\kappa\in C^\infty_0(-\infty,T)$, we obtain $\zeta_\tau$, $\nabla\zeta_\tau, \partial_t \zeta_\tau \in  L^2((0,T)\times J^\ve)$ for $J=S,P$. The Lebesgue differentiation theorem and  the regularity of $h_S \in  L^2(0,T;V(S^\ve))$, $h_P -a \in  L^2(0,T;V(P^\ve))$,  $v^0 \in V(S^\ve)$, and  $w^0 -a \in V(P^\ve)$,  imply $\zeta_\tau \to \zeta$ in $L^2(0,T;V(J^\ve))$ and, by applying the trace theorem, also in $L^2((0,T)\times\Gamma^\ve_{P})$ and $L^2((0,T)\times\Gamma^\ve_{J,0})$, for $J=S,P$, as~$\tau\to0$. Hence, testing~\eqref{weak_formulation_soil} and \eqref{weak_formulation_plant} with the corresponding $\zeta_\tau$ and  taking the limit as $\tau \to 0$ yield the inequality stated in the lemma.
\end{proof}

\begin{remark}
 Notice that in the proof of Theorem~\ref{theorem_microscopic_models_uniqueness} we  use the inequalities in Lemma~\ref{lemma_uniqueness_lemma_1}  for  $\kappa\in C^\infty_0(0,T)$, yielding
 $$
 \int_0^T\!\! \Big\langle \eta_{S,\delta}^+\big(h_{S,0},v^0\big), \frac{d\kappa}{dt}\Big\rangle_{S^\ve} =  \Big\langle \eta_{S,\delta}^+\big(h_{S,0},v^0\big), \kappa(T) - \kappa(0) \Big\rangle_{S^\ve}=0.
 $$
 The same holds   for   terms involving $\eta_{P,\delta}^+\big(h_{P,0},w^0\big)$, $\eta_{S,\delta}^-\big(h_{S,0},v^0\big)$, and $\eta_{P,\delta}^-\big(h_{P,0},w^0\big)$.
\end{remark}

\begin{proof}[Proof of Theorem~\ref{theorem_microscopic_models_uniqueness}]
	To prove the uniqueness result, assume there are two solutions~$(h_{S,1},h_{P,1})$ and~$(h_{S,2},h_{P,2})$,  consider a doubling of the time variable $(t_1,t_2)\in(0,T)^2$ and define functions~$h_{J,1}, h_{J,2}: (0,T)^2\times J^\ve \to \mathbb{R}$ such that
$$
				h_{J,1}(x,t_1,t_2) = h_{J,1}(x,t_1)\quad\text{and}\quad h_{J,2}(x,t_1,t_2) = h_{J,2}(x,t_2), \quad \text{ for } \; J=S,P.
$$
Considering \eqref{inequality_uniqueness_lemma_1} with  $v^0=h_{S,2}(t_2)$,~$w^0=h_{P,2}(t_2)$, and~$\kappa(t_2):t_1\to\kappa(t_1,t_2)$, for non-negative  $\kappa\in C_0^\infty((0,T)^2)$, we obtain
\begin{equation}\label{inequality_t2_constant_uniqueness_proof}
\begin{aligned}
& \int_0^T\!\!\Big[-\Big(\int_{S^\ve} \!\! \eta_{S,\delta}^+(h_{S,1},h_{S,2}(t_2))dx + \int_{P^\ve}\!\!\eta_{P,\delta}^+(h_{P,1},h_{P,2}(t_2)) dx \Big)\frac{d\kappa(t_2)}{dt_1} \\
				& \qquad +
				\big[\big \langle K_S(h_{S,1})\big(\nabla h_{S,1} + {e}_3\big),
				\nabla \sigma_\delta^+(h_{S,1}-h_{S,2}(t_2)) \big\rangle_{S^\ve} \\
				& \qquad \quad +
				\big \langle H_r^\ve K_P(h_{P,1})\big(\nabla h_{P,1} + {e}_3\big),
				\nabla \sigma_\delta^+(h_{P,1}-h_{P,2}(t_2)) \big\rangle_{P^\ve} \big]\kappa(t_2)\\
				&  +
				\ve  k_\Gamma  \big \langle h_{S,1}-h_{P,1},
				\sigma_\delta^+(h_{S,1}-h_{S,2}(t_2))-
				\sigma_\delta^+(h_{P,1}-h_{P,2}(t_2)) \big\rangle_{\Gamma_P^\ve} \kappa(t_2)\\
				& \qquad +
				\big[\big \langle f(h_{S,1}),\sigma_\delta^+(h_{S,1}-h_{S,2}(t_2))\big\rangle_{\Gamma^\ve_{S,0}}
				 \\ & \qquad \qquad +
				\big \langle\mathcal{T}_\text{pot},\sigma_\delta^+(h_{P,1}-h_{P,2}(t_2)) \big \rangle_{\Gamma^\ve_{P,0}} \big]\kappa(t_2) \Big ] dt_1
				\leq 0,			
			\end{aligned}
		\end{equation}
for a.e.~$t_2\in(0,T)$, whereas  considering   $v^0=h_{S,1}(t_1)$, $w^0=~h_{P,1}(t_1)$, and~$\kappa(t_1):t_2\to\kappa(t_1,t_2)$ in the inequality with $\sigma_\delta^-$ and  $\eta^-_{J, \delta}$, for $J=S,P$,    yields
\begin{equation}\label{inequality_t1_constant_uniqueness_proof}
\begin{aligned}
&\int_0^T\!\!\Big[-\Big(\int_{S^\ve}\!\!\eta_{S,\delta}^-(h_{S,2},h_{S,1}(t_1)) dx  + \int_{P^\ve}\!\!\eta_{P,\delta}^-(h_{P,2},h_{P,1}(t_1)) dx \Big) \frac{d\kappa(t_1)}{dt_2}
\\&
\qquad +  \big[ \big \langle  K_S(h_{S,2})\big(\nabla h_{S,2} + {e}_3\big),\nabla\sigma_\delta^-(h_{S,2}-h_{S,1}(t_1)) \big \rangle_{S^\ve}
\\& \qquad \quad  +  \big \langle H_r^\ve  K_P(h_{P,2})\big(\nabla h_{P,2} + {e}_3\big),\nabla\sigma_\delta^-(h_{P,2}-h_{P,1}(t_1)) \big \rangle_{P^\ve} \big]\, \kappa(t_1)\\
&  + \ve k_\Gamma  \big\langle h_{S,2}-h_{P,2}, \sigma_\delta^-(h_{S,2} - h_{S,1}(t_1))-\sigma_\delta^-(h_{P,2} -h_{P,1}(t_1)) \big\rangle_{\Gamma_P^\ve} \kappa(t_1)\\
&\qquad  + \big[\big\langle f(h_{S,2}), \sigma_\delta^-(h_{S,2} -h_{S,1}(t_1)) \big\rangle_{\Gamma^\ve_{S,0}}
\\& \qquad \qquad + \big\langle\mathcal{T}_\text{pot},\sigma_\delta^-(h_{P,2}-h_{P,1}(t_1)) \big\rangle_{\Gamma^\ve_{P,0}}\big]\kappa(t_1) \Big] dt_2 \leq 0,
\end{aligned}
\end{equation}
for a.e.~$t_1\in(0,T)$. Integrating~\eqref{inequality_t2_constant_uniqueness_proof} with respect to~$t_2$ and~\eqref{inequality_t1_constant_uniqueness_proof} with respect to~$t_1$, adding the resultant inequalities and using~$\sigma_{\delta}^-(z)=-\sigma_{\delta}^+(-z)$, we obtain
\begin{eqnarray}\label{uniqueness_inequality_doubled_time}
-\!\! \sum_{J=S,P}\!\!\!\big\langle \eta_{J,\delta}^+(h_{J,1},h_{J,2})\partial_{t_1}\kappa + \eta_{J,\delta}^-(h_{J,2},h_{J,1})\partial_{t_2}\kappa, 1 \big\rangle_{Q_J} \nonumber \\
+ \big\langle K_S(h_{S,1})\big(\nabla h_{S,1} + {e}_3\big) - K_S(h_{S,2})\big(\nabla h_{S,2} + {e}_3\big),   \nabla \sigma_\delta^+(h_{S,1}- h_{S,2})\kappa \big\rangle_{Q_S}\nonumber\\
	+ \big\langle H_r^\ve \big[K_P(h_{P,1})(\nabla h_{P,1} + {e}_3) - K_P(h_{P,2})(\nabla h_{P,2} + {e}_3)\big],  \nabla \sigma_\delta^+(h_{P,1}-h_{P,2})\kappa\big \rangle_{Q_P} \nonumber\\
+ \ve k_\Gamma \big \langle (h_{S,1}-h_{S,2})-(h_{P,1}-h_{P,2}), \big[\sigma_\delta^+(h_{S,1}-h_{S,2}) -\sigma_\delta^+(h_{P,1}-h_{P,2})\big]\kappa \big\rangle_{\tilde \Gamma_P^\ve} \nonumber\\
+ \big \langle f(h_{S,1})-f(h_{S,2}), \sigma_\delta^+(h_{S,1}-h_{S,2})\kappa \big \rangle_{(0,T)^2 \times \Gamma^\ve_{S,0}}\leq 0,
\end{eqnarray}
where $Q_S = (0,T)^2 \times S^\ve$,  $Q_P = (0,T)^2 \times P^\ve$, and $\tilde \Gamma_P^\ve = (0,T)^2 \times \Gamma_P^\ve$.
Using the definition of $\sigma^+_\delta$ we have
\begin{equation}\label{inequality_uniqueness_lower_bound_flux_term_double_time}
\begin{aligned}
&\big[K_S(h_{S,1})(\nabla h_{S,1} + {e}_3)-K_S(h_{S,2})(\nabla h_{S,2} + {e}_3)\big]\cdot \nabla\sigma_\delta^+(h_{S,1}-h_{S,2})\\
&\quad \geq (K_S(h_{S,1}) - \varsigma)  \big|\nabla (h_{S,1}-h_{S,2})\big|^2 \big(\sigma_\delta^+(h_{S,1}-h_{S,2})\big)^\prime \\
& \qquad - \frac 1 {2\varsigma} \big[1+|\nabla h_{S,2}|^2\big] \big| K_S(h_{S,1})-K_S(h_{S,2})\big|^2\ \big(\sigma_\delta^+(h_{S,1}-h_{S,2})\big)^\prime,
\end{aligned}
\end{equation}
where $\big(\sigma_\delta^+\big)'$ is non-negative and  singular as~$\delta\downarrow 0$.  The first term  on the right-hand side of~\eqref{inequality_uniqueness_lower_bound_flux_term_double_time}   is non-negative for  $0<\varsigma\leq  \min\{ K_{R,0}, K_{B,0}\}$  and the second term  is  non-zero only  if $0<h_{S,1}-h_{S,2}<\delta$  with
$$
\big| K_S(h_{S,1})-K_S(h_{S,2})\big|^2\big(\sigma_\delta^+(h_{S,1}- h_{S,2})\big)^\prime = \frac{| K_S(h_{S,1})-K_S(h_{S,2})|^2}{\delta} \leq L^2_{K_S} \delta,
$$
where $L_{K_S} = \max \{L_{K_R}, L_{K_B}\}$ and $L_{K_J}$ is the Lipschitz constant for $K_J$, with $J=R,B$.
A similar estimate holds for $H^\ve_r[K_P(h_{P,1})(\nabla h_{P,1} + {e}_3) - K_P(h_{P,2})(\nabla h_{P,2} + {e}_3)] \cdot  \nabla \sigma_\delta^+(h_{P,1}-h_{P,2})$.  The definition of $\sigma^+_\delta$ yields
\begin{equation*}
\begin{aligned}
&\big[(h_{S,1}-h_{S,2})-(h_{P,1}-h_{P,2})\big]\big[\sigma_\delta^+(h_{S,1}-h_{S,2}) - \sigma_\delta^+(h_{P,1}-h_{P,2})\big]\\
&\quad  \geq \big[(h_{S,1}-h_{S,2})^+-(h_{P,1}-h_{P,2})^+\big]\big[\sigma_\delta^+(h_{S,1}- h_{S,2}) - \sigma_\delta^+(h_{P,1}-h_{P,2})\big] = I,
\end{aligned}
\end{equation*}
where $u^+ = \max\{ u, 0 \}$.
When  $h_{S,1}-h_{S,2}\leq0$, $h_{P,1}-h_{P,2}\leq0$,  or $h_{S,1} - h_{S,2} \geq \delta$ and $ h_{P,1}- h_{P,2} \geq \delta$, we directly have  $I\geq 0$. If    $h_{S,1}-h_{S,2} \geq \delta$ and  $0<h_{P,1}-h_{P,2} < \delta$, then $\sigma_\delta^+(h_{S,1}-h_{S,2}) = 1$ and $ \sigma_\delta^+(h_{P,1}-h_{P,2}) < 1$, resulting in
$$
I \geq \big[(h_{S,1}-h_{S,2})^+-(h_{P,1}-h_{P,2})^+\big] - \big[(h_{S,1}-h_{S,2})^+-(h_{P,1}-h_{P,2})^+\big]= 0.
$$
In the case where  $\delta > h_{S,1}-h_{S,2} > h_{P,1}-h_{P,2} > 0$ we have $\sigma_\delta^+(h_{S,1}- h_{S,2}) > \sigma_\delta^+(h_{P,1} -h_{P,2})$ and hence $I\geq 0$. Analogous arguments imply  $I\geq  0$  when $h_{P,1}-h_{P,2} > h_{S,1}-h_{S,2} > 0$.

Using \eqref{uniqueness_inequality_doubled_time}, the estimates above, and the fact that $f$ is non-decreasing, yields
\begin{equation}\label{uniqueness_inequality2_doubled_time}
\begin{aligned}
- \sum_{J=S,P}\Big[\big \langle\eta_{J,\delta}^+(h_{J,1},h_{J,2})\partial_{t_1}\kappa + \eta_{J,\delta}^-(h_{J,2},h_{J,1})\partial_{t_2}\kappa, 1 \big \rangle_{Q_J}
\\ + \delta C \big(1+\|\nabla h_{J,2}\,  \kappa\|^2_{L^2(Q_J)} \big)\Big]& \leq 0.
\end{aligned}
\end{equation}
The monotonicity of~$ {\theta}_S$ and $\sigma_\delta^+$  and the non-negativity of~$\sigma_\delta^+$ ensure
\begin{equation*}
\Big\|\int_{h_{S,2}}^{h_{S,1}}\!\!\! \!\! \sigma_\delta^+(z-h_{S,2}) {\theta}_S^{'}(z)dz\Big\|^p_{L^p(Q_S)}
\!\!\! \leq \big\|\sigma_\delta^+(h_{S,1}-h_{S,2})( {\theta}_S(h_{S,1}) -  {\theta}_S(h_{S,2}))\big\|^p_{L^p(Q_S)},
\end{equation*}
for $1<p<\infty$. A similar estimate is obtained for $\sigma^{-}_\delta$ by using  that~$\sigma^-_\delta(z)=-\sigma^+_\delta(-z)$.  The boundedness of $\theta_S$ and $\sigma^+_\delta$ implies  $ \eta_{S,\delta}^+(h_{S,1},h_{S,2})$ and  $\eta_{S,\delta}^-(h_{S,2},h_{S,1})$ are uniformly bounded  in $L^p(Q_S)$ and, since~$0\leq \sigma_\delta^+\leq 1$, we have
$$
|\eta_{S,\delta}^+(h_{S,1},h_{S,2})| \leq ( {\theta}_S(h_{S,1}) -  {\theta}_S(h_{S,2}))^+, \; \; |\eta_{S,\delta}^-(h_{S,2},h_{S,1})| \leq ( {\theta}_S(h_{S,1}) -  {\theta}_S(h_{S,2}))^+.
$$
Additionally, with $(t_1, t_2, x) \in Q_S$, we have $({\theta}_S(h_{S,1}(t_1,x)) -  {\theta}_S(h_{S,2}(t_2,x)))^+ \!=\! 0$ if   $h_{S,1}(t_1,x)\leq h_{S,2}(t_2,x)$, and $\sigma_\delta^+(z-h_{S,2}(t_2,x)) = 0$ for  $z\in\big(h_{S,1}(t_1,x),h_{S,2}(t_2,x)\big)$, $ \sigma_\delta^+(z - h_{S,2}(t_2,x)) = 1$  for  $z\geq h_{S,2}+\delta(t_2, z)$. Then the Lipschitz continuity of ${\theta}_S$ implies
\begin{equation*} 
\begin{aligned}
\big|\eta_{S,\delta}^+(h_{S,1}(t_1,x),h_{S,2}(t_2,x))-\big( {\theta}_S(h_{S,1}(t_1,x)) -  {\theta}_S(h_{S,2}(t_2,x))\big)^+\big|\\
\quad \leq L_{ {\theta}_S}	\Big|\int_{h_{S,2}(t_2,x)}^{h_{S,1}(t_1,x)}\big[\sigma_\delta^+(z -h_{S,2}(t_2,x))-1 \big] dz\Big| & \leq L_{ {\theta}_S} \delta.
\end{aligned}
\end{equation*}
Using similar arguments for $h_{P,1}$ and $h_{P,2}$ and the fact that~$\sigma^-(z)=-\sigma^+(-z)$ yields
\begin{equation}\label{eta_plus_convergence_ae}
\begin{aligned}
&  \eta_{J,\delta}^+(h_{J,1},h_{J,2})\to \big( {\theta}_J(h_{J,1}) -  {\theta}_J(h_{J,2})\big)^+, \\
&   \eta_{J,\delta}^-(h_{J,2},h_{J,1})\to \big( {\theta}_J(h_{J,1}) -  {\theta}_J(h_{J,2})\big)^+,
\end{aligned}
\end{equation}
pointwise a.e.~in $Q_J$ as $\delta \to 0$, for $J=S,P$. The dominated convergence theorem implies the strong convergence of $\eta_{J,\delta}^+(h_{J,1},h_{J,2})$ and $\eta_{J,\delta}^-(h_{J,2},h_{J,1})$
in $L^p((0,T)^2\times J^\ve)$, as~$\delta\to0$, for $J=S,P$ and $1<p< \infty$. Taking in~\eqref{uniqueness_inequality2_doubled_time} the limits as $\delta\to 0$  we obtain
\begin{equation}\label{uniqueness_inequality3_doubled_time}
-\Big\langle \! \sum_{J=S,P}\int_{J^\ve}\!\!\big( {\theta}_J(h_{J,1}) -  {\theta}_J(h_{J,2})\big)^+ dx,
\partial_{t_1}\kappa + \partial_{t_2}\kappa \Big\rangle_{(0,T)^2} \leq 0.
\end{equation}
For any non-negative~$\kappa\in C_0^\infty(0,T)$,  extended by zero for $t\leq 0$ and $t\geq T$,   there exists~$\varrho^*\in(0, T/2)$ such that $\kappa(t) = 0$ if  $0<t<\varrho^*$ or $t>T-\varrho^*$. For a non-negative~$\vartheta\in C_0^\infty(\mathbb{R})$ of unit mass we have that $\vartheta(r/{\varrho}) = 0$, for $\lvert r\rvert\geq\varrho$ and some $\varrho>0$, and  the function
\begin{equation}\label{kappa_varrho}
\kappa_{\varrho}(t_1,t_2):=\frac{1}{\varrho}\vartheta\Big(\frac{t_1-t_2}{\varrho}\Big)\kappa\Big(\frac{t_1 + t_2}{2}\Big),
\end{equation}
is admissible in~\eqref{uniqueness_inequality3_doubled_time}, for $\varrho<2\varrho^*$.
Applying the change of variables~$\tau = t_1-t_2$ and denoting $t=t_1$, yield
\begin{equation}\label{uniqueness_inequality4_doubled_time_changed_variable}
- \int_{\mathbb{R}}\! \frac{1}{\varrho}\vartheta\Big(\frac{\tau}{\varrho}\Big)\int_{0}^T\! \!\!\sum_{J=S,P}\int_{J^\ve}\!\!\!\big( {\theta}_J(h_{J,1}) -  {\theta}_J(h_{J,2}^\tau)\big)^+ dx \, \partial_t\kappa^{\frac \tau 2} \, dtd\tau \leq 0,
\end{equation}
with the abbreviation~$h_J^\tau(t) = h_J(t-\tau)$ for $J=S,P$. To take the limit  as~$\tau\to0$, we first show $ {\theta}_J(h_{J,2}^\tau)\to {\theta}_J(h_{J,2})$ strongly in~$L^2((0,T)\times J^\ve)$, where $J=S,P$. Assume~$\tau>0$ and consider
$$\zeta(t):=(h_{S,2}(t)-h_{S,2}(t-\tau))\chi_{[0,T)}(t), \quad \text{ and } \quad
\zeta_\tau(t):=\frac{1}{\tau}\int_t^{t+\tau}\zeta(s)ds.
$$
The regularity of $h_{S,2}$ ensures that $\zeta_\tau$ is an admissible test function in~\eqref{weak_formulation_soil}.
Using an integration by parts  and the regularity of~$\zeta_\tau$, we obtain
\begin{equation*}
\big\langle\partial_t {\theta}_S(h_{S,2}), \zeta_\tau\big\rangle_{V(S^\ve)^\prime, T}
= \Big\langle\frac{1}{\tau}\big( {\theta}_S(h_{S,2})- {\theta}_S(h_{S,2}(\cdot-\tau))\big),h_{S,2}-h_{S,2}(\cdot-\tau)\Big\rangle_{S^\ve_T},
\end{equation*}
where we used that $ {\theta}_S(h_{S,2}(t))= {\theta}_S(h_{S,0})$ for~$t<0$ and $\chi_{[0,T)}(t)=0$ for~$t\geq T$.	An analogous argument holds also for~$\tau<0$. Thus from equation~\eqref{weak_formulation_soil} we have
\begin{equation}\label{uniqueness_inequality_for_convergence_of_h_minus_tau}
\begin{aligned}
&\big \langle  {\theta}_S(h_{S,2})- {\theta}_S(h_{S,2}^\tau), h_{S,2}-h_{S,2}^\tau \rangle_{S^\ve_T} \leq
\tau\Big[\big| \big \langle K_S(h_{S,2})(\nabla h_{S,2} + {e}_3), \nabla \zeta_\tau \big \rangle_{S^\ve_T}\big| \\
& + \ve \, k_\Gamma \, \big( \big|\langle h_{S,2}, \zeta_\tau \rangle_{\Gamma^\ve_{P,T}} \big| +
\big|\langle  h_{P,2}, \zeta_\tau\rangle_{\Gamma^\ve_{P,T} }\big| \big) +
\big | \big \langle f(h_{S,2}), \zeta_\tau \big \rangle_{(0,T)\times \Gamma^\ve_{S,0}} \big|\Big].
\end{aligned}
\end{equation}
By the Lebesgue differentiation theorem, we have $\zeta_\tau \to \zeta$ strongly in $L^2(0,T;V(S^\ve))$, and by the trace theorem, also  in $L^2((0,T)\times\Gamma^\ve_{P})$ and $L^2((0,T)\times\Gamma^\ve_{S,0})$.
Hence  the right hand side of~\eqref{uniqueness_inequality_for_convergence_of_h_minus_tau}	converges to $0$ as~$\tau\to0$ and from the Lipschitz continuity and monotonicity of~$ {\theta}_S$ we obtain
\begin{equation}\label{uniqueness_strong_convergence_h_tau_h_soil}
\big\| {\theta}_S(h_{S,2})- {\theta}_S(h_{S,2}^\tau)\big\|^2_{L^2((0,T)\times S^\ve)} \!\! \leq
C\big\langle \theta_S(h_{S,2})- \theta_S(h_{S,2}^\tau), h_{S,2}-h_{S,2}^\tau \big \rangle_{S^\ve_T} \!\! \to 0,
\end{equation}
as $\tau \to 0$. An identical argument is employed to show
\begin{equation}\label{uniqueness_strong_convergence_h_tau_h_root}
\big\| {\theta}_P(h_{P,2})- {\theta}_P(h_{P,2}^\tau)\big\|_{L^2((0,T)\times P^\ve)} \to  0 \quad \text{ as } \tau  \to 0.
\end{equation}
Combining~\eqref{uniqueness_strong_convergence_h_tau_h_soil} and~\eqref{uniqueness_strong_convergence_h_tau_h_root} with the continuity of~$\partial_t\kappa$ and taking~$\tau \to0$ in~\eqref{uniqueness_inequality4_doubled_time_changed_variable} imply
\begin{equation}\label{uniqueness_inequality5}
- \int_{0}^T  \sum_{J=S,P}\int_{J^\ve}\big( {\theta}_J(h_{J,1}) -  {\theta}_J(h_{J,2})\big)^+ dx \, \partial_t\kappa(t) \, dt \leq 0.
\end{equation}
Applying integration by parts and using  that~$\kappa$ is compactly supported  yield
\begin{equation*}
\int_{0}^T\!\!\!\kappa(t)\,  \partial_t\Big[\int_{S^\ve} \!\!\big( {\theta}_S(h_{S,1}) -  {\theta}_S(h_{S,2})\big)^+ dx
+ \int_{P^\ve} \!\!\big( {\theta}_P(h_{P,1}) -  {\theta}_P(h_{P,2})\big)^+ dx\Big]dt \leq 0,
\end{equation*}
and, since this holds for all non-negative~$\kappa\in C_0^\infty(0,T)$, it follows that
\begin{equation*}
\begin{aligned}
& \sum_{J=S,P}\int_{J^\ve}\!\!\big( \theta_J(h_{J,1}(t,x)) -  \theta_J(h_{J,2}(t,x))\big)^+ dx
\\
& \qquad \leq \sum_{J=S,P}\int_{J^\ve}\!\!\big( \theta_J(h_{J,1}(0,x)) -  \theta_J(h_{J,2}(0,x))\big)^+ dx =0,
\end{aligned}
\end{equation*}
for a.e.~$t\in(0,T)$ and  $\big( {\theta}_J(h_{J,1}) -  {\theta}_J(h_{J,2})\big)^+=0$ a.e.~in~$(0,T)\times J^\ve$, for $J=S,P$.
Using~\eqref{inequality_uniqueness_lemma_1}  with $v^0=h_{S,1}(t_1),w^0=h_{P,1}(t_1)$,  and~$\kappa(t_1)$, and the corresponding inequality for $\sigma_\delta^-$ and $\eta^-_{J, \delta}$  with~$v^0=h_{S,2}(t_2),w^0=h_{P,2}(t_2)$ and~$\kappa(t_2)$ and performing the same calculations as above yield $\big( \theta_J(h_{J,2}) -  \theta_J(h_{J,1})\big)^+=0$  a.e.~in~$(0,T)\times J^\ve$, for $J=S,P$. Thus, since ${\theta}_S$ and ${\theta}_P$ are strictly increasing, we obtain the uniqueness of~$h_S$ and~$h_P$.
\end{proof}

\begin{proposition}\label{proposition_non-positivity proof}
Under Assumption~\ref{assumption} and additionally if~$K_J(z)= K_{J}(0)$ for~$z\geq0$, where $J=R,B,P$,   and $f(z)\geq 0$ for~$z\geq 0$,   solutions of~\eqref{richards_soil}--\eqref{bcs_root}  are non-positive.
\end{proposition}

\begin{proof}
Using $h_S^+$ and~$h_P^+$ as test functions in~\eqref{weak_formulation_soil} and~\eqref{weak_formulation_plant} and  adding the resulting equations yield
\begin{equation}\label{test_weak_form_h_plus}
\begin{aligned}
& \big[ \big\langle\partial_t {\theta}_S\big(h_{S}\big),h_S^+\big\rangle_{V(S^\ve)^\prime, T}  + \langle K_S(h_S)\big(\nabla h_{S} + {e}_3\big), \nabla{h_S^+}\rangle_{S^\ve_T} \big]\\
& \quad +  \big[ \big\langle\partial_t {\theta}_P\big(h_{P}\big),h_P^+\big\rangle_{V(P^\ve)^\prime, T}  + \langle H_r^\ve K_P(h_P)\big(\nabla h_{P} + {e}_3\big), \nabla{h_P^+}\rangle_{P^\ve_T} \big]\\
&  + \ve  k_\Gamma  \langle h_S- h_P, h_S^+ - h_P^+ \rangle_{\Gamma_{P,T}^\ve} + \langle f(h_S), h_S^+ \rangle_{\Gamma_{S,0, T}^\ve} + \langle \mathcal{T}_\text{pot}, h_P^+ \rangle_{\Gamma_{P,0, T}^\ve} =0.
\end{aligned}
\end{equation}
Notice that $h_P^+\in$~$V(P^\ve)$, since $h_P = a\leq0$ on $\Gamma_{P,L_3}$. We first show
\begin{equation}\label{eqn_upper_bound_psi result}
\Psi_J(h_J(s))	-
\Psi_J(h_J(s-\tau))\leq\big[\theta_J\big(h_J(s)\big)- \theta_J\big(h_J(s-\tau)\big)\big]h_J^+(s),
\end{equation}
for all~$s\in(0,T]$ and $J=S,P$, where $\Psi_J:\mathbb{R}\to[0,\infty)$ is given by $\Psi_J(h_J) = \int_0^{h_J} {\theta}'_J(z)z^+dz$. Integrating by parts and using $ {\theta}_J(0) = 0$ yield
$$
\Psi_J(h_J) = \begin{cases}
{\theta}_J(h_J)h_J - \int_0^{h_J} {\theta}_J(z)dz &\text{ if }h_J>0,\\
0  &\text{ otherwise}.
\end{cases}
$$
For~$\tau>0$ and~$s\in(0,T]$,  if~$h_J(s)>h_J(s-\tau)>0$, where~$h_J(s-\tau)=h_{J,0}$ for~$\tau>s$, we have
\begin{equation}\label{eqn1_upper_bound_Psi}
\begin{aligned}
&\Psi_J(h_J(s))	- \Psi_J(h_J(s-\tau)) \leq \theta_J(h_{J}(s))h_J(s) -  \theta_J(h_{J}(s-\tau))h_J(s-\tau)
\\ &- (h_{J}(s)-h_{J}(s-\tau)) \theta_J(h_{J}(s-\tau))\leq \big( \theta_J(h_{J}(s))- \theta_J(h_{J}(s-\tau))\big)h_J^+(s).
\end{aligned}
\end{equation}
Following a similar line of argument, result~\eqref{eqn1_upper_bound_Psi} is also obtained for~$h_J(s-\tau)>h_J(s)>0$. If~$h_J(s)>0>h_J(s-\tau)$, then since~$\theta_J(h_J(s-\tau))<0$ it follows
$$
\Psi_J(h_J(s)) - \Psi_J(h_J(s-\tau)) \leq \theta_J\big(h_{J}(s)\big)h_J(s) \leq \big( {\theta}_J(h_{J}(s))- {\theta}_J(h_{J}(s-\tau))\big)h_J^+(s).
$$
If~$h_J(s-\tau)>0>h_J(s)$, then since~$\Psi_J(h_J(s)) = 0$ and~$h_J^+(s)=0$ we have
$$
\Psi_J(h_J(s)) -\Psi_J(h_J(s-\tau)) \leq 0 =\big[\theta_J(h_J(s))- \theta_J(h_J(s-\tau))\big]h_J^+(s).
$$
Combining estimates above yields \eqref{eqn_upper_bound_psi result}. Multiplying each side of~\eqref{eqn_upper_bound_psi result} by~$1/\tau$, using that $\partial_t {\theta}_J(h_{J})\in L^2(0,T;V(J^\ve)^\prime)$ and $h_J^+\in L^2\big(0,T;V(J^\ve)\big)$ and that the initial condition~$h_{J,0}\leq0$ implies~$\Psi_J(h_{J,0})=0$, and taking the limit as $\tau \to 0$ yields
\begin{equation}\label{eqn4_result_positive_time_derivative_soil}
\int_0^T\big\langle\partial_t {\theta}_J(h_{J}),h_J^+\big\rangle_{V(J^\ve)^\prime}\, dt \geq \int_{J^\ve}\Psi_J(h_J(T))\, dx \geq 0, \qquad \text{ for } \; \;  J=S,P.
\end{equation}
The definition of~$h_S^+,h_S^-,h_P^+$ and~$h_P^-$, implies
\begin{equation*}
h_Sh_S^+ + h_Ph_P^+ - h_Ph_S^+ - h_Sh_P^+ = \big(h_S^+-h_P^+\big)^2 - h_P^-h_S^+ -h_S^-h_P^+ \geq  0.
\end{equation*}
Combining the results above, together  with $\mathcal{T}_\text{pot}>0$ and $f\big(h_S\big)h_S^+\geq0$, it follows, from~\eqref{test_weak_form_h_plus},
\begin{equation}\label{eqn2_h_non_positive}
\int_0^T\!\!\Big[\!\int_{S^\ve} \!\!\!\! K_S(h_S)\big[|\nabla h_{S}^+|^2 + \partial_{x_3} h_{S}^+\big]dx + \!\int_{P^\ve} \!\!\!\! H_r^\ve K_P(h_P)\big[|\nabla h_{P}^+|^2 + \partial_{x_3} h_{P}^+\big]dx\Big]dt \leq 0.
\end{equation}
Using $K_J(h_J)= K_{J}(0)>0$ for $h_J\geq 0$ and  the boundary conditions on $\Gamma^\ve_{J, L_3}$, gives
\begin{equation}\label{eqn3_h_non_positive}
\int_0^T\!\!\!\int_{J^\ve} \!\!\! K_J(h_J)\partial_{x_3} h_{J}^+dxdt= K_{J}(0)\int_0^T\!\!\!\int_{J^\ve}\!\!\! \partial_{x_3} h_{J}^+dxdt = K_{J}(0)\int_0^T\!\!\! \int_{\Gamma_{J, 0}^\ve} \!\!\! h_{J}^+dxdt,
\end{equation}
for $J=S,P$.	Thus from~\eqref{eqn2_h_non_positive} and \eqref{eqn3_h_non_positive} we obtain
\begin{equation*}\label{eqn5_h_non_positive}
 \int_0^T\!\!\!\int_{J^\ve} \!\! K_J\big(h_J\big)\big\lvert\nabla h_{J}^+\big\rvert^2 dxdt = 0, \qquad
K_{J}(0)\int_0^T\!\!\!\int_{\Gamma_{J, 0}^\ve}\!\!\! h_{J}^+dxdt = 0, \qquad J=S,P,
\end{equation*}
and  that $h_S$ and~$h_P$ are non-positive over~$(0,T)\times S^\ve$ and~$(0,T)\times P^\ve$ respectively.
\end{proof}
\begin{remark}
Theorems~\ref{theorem_microscopic_models_existence} and~\ref{theorem_microscopic_models_uniqueness} were proven for $h_S:(0,T)\times S^\ve\to\mathbb{R}$ and~$h_P:(0,T)\times P^\ve\to\mathbb{R}$. Physically realistic functions for water content~$\theta_R$, $\theta_B$, and~$\theta_P$ and hydraulic conductivity~$K_R$, $K_B$, and~$K_P$ are  usually not defined for positive values of pressure head and, in the cases where they are, they often fail to satisfy Assumption~\ref{assumption}, see~\citep{van1980closed,janott2011one,bittner2012functional}. In the Appendix, we provide functions for~$\theta_S$,~$K_S$,~$\theta_P$ and~$K_P$, which extend the expressions used in~\citep{van1980closed,janott2011one,bittner2012functional}, to positive values of pressure head, in a way that the criteria of Assumption~\ref{assumption}, Theorem~\ref{theorem_microscopic_models_uniqueness} and Proposition~\ref{proposition_non-positivity proof} are satisfied. These extensions can therefore be assumed throughout the proofs of Theorems~\ref{theorem_microscopic_models_existence} and~\ref{theorem_microscopic_models_uniqueness}. Moreover, since Proposition~\ref{proposition_non-positivity proof} shows that $h_S$ and~$h_P$ will remain non-positive, for any non-positive initial condition, the question of how realistic these extensions are for positive values of pressure head is not a concern.
\end{remark}

\section{Derivation of macroscopic model}\label{section:macro_model}
To derive macroscopic equations from the microscopic model for the water transport in vegetated soil, we  apply the two-scale convergence and periodic unfolding method, see e.g.~\cite{Allaire_1992,  cioranescu2018periodic,  Neuss-Radu_1996, Nguetseng_1989}. To pass to the limit as $\ve \to 0$  we first derive a priori estimates uniform in $\ve >0$.

\begin{lemma}\label{lemma:two-scale_convergence_1}
	Under Assumption~\ref{assumption}, solutions~$(h_S^\varepsilon, h_P^\varepsilon)$ of~\eqref{richards_soil}--\eqref{bcs_root}  satisfy the following estimates
\begin{equation}\label{apriori_1_1}
\begin{aligned}
\big\|h_S^\varepsilon\big\|^2_{L^2(0,T;V(S^\varepsilon))}  + \big\|h_P^\varepsilon\big\|^2_{L^2( P^\varepsilon_T)} +
\big\|\partial_{x_3}h_P^\varepsilon\big\|^2_{L^2(P^\varepsilon_T)} +
\varepsilon\big\|\nabla_{\hat{x}} h_P^\varepsilon\big\|^2_{L^2(P^\varepsilon_T)} & \leq C, \\
\ve \|h_S^\ve\|^2_{L^2(\Gamma_{P,T}^\ve)}  + \ve \|h_P^\ve\|^2_{L^2(\Gamma_{P,T}^\ve)} +  \|h_S^\ve\|^2_{L^2(\Gamma_{S,0, T}^\ve)} +  \|h_P^\ve\|^2_{L^2(\Gamma_{P, 0, T}^\ve)} & \leq C, \\
\sup_{0\leq t\leq T}\| {\Theta}_S^\varepsilon(\cdot,{h}^\ve_{S}(t))\|_{L^1(S^\ve)} +
\sup_{0\leq t\leq T}\|	 {\Theta}_P(h^\ve_{P}(t)) \|_{L^1(P^\ve)} & \leq C,
\end{aligned}
\end{equation}
where~$\nabla_{\hat{x}}h_P^\varepsilon=(\partial_{x_1}h_P^\varepsilon,\partial_{x_2}h_P^\varepsilon)^\top$,~$ {\Theta}_S^\varepsilon$ and~$ {\Theta}_P$    as in~\eqref{big_theta}, and~$C>0$ is independent of $\ve$.
\end{lemma}	
\begin{proof}
Considering $h_S^\varepsilon$ and $h_P^\varepsilon-a$ as test functions in \eqref{weak_formulation_soil} and \eqref{weak_formulation_plant}, respectively, adding the resulting equations, using Assumption~\ref{assumption} on~$f$ and $K_J$, for $J=B,R,P$, and  $a\leq 0$, yields
\begin{equation}\label{eqn1_lemma_2scale_convergence_1}
\begin{aligned}
&\langle\partial_t {\theta}_S^\varepsilon(x, h_S^\varepsilon), h_S^\varepsilon \rangle_{V(S^\ve)^\prime, \tau} +\langle\partial_t {\theta}_P(h_P^\varepsilon), h_P^\varepsilon-a\rangle_{V(P^\ve)^\prime, \tau}  \\
& +  \frac{K_{S,0}}2\|\nabla h_S^\varepsilon\|^2_{L^2(S^\ve_\tau)} +
\frac{\tilde K_{P,0}}2\| I_{\sqrt{\varepsilon}}\nabla h_P^\varepsilon\|^2_{L^2(P^\ve_\tau)}  +
\varepsilon k_\Gamma \| h_S^\varepsilon-h_P^\varepsilon\|^2_{L^2(\Gamma^\ve_{P,\tau})}
\\&	\leq
\delta \big[	\|h_S^\varepsilon\|^2_{L^2(\Gamma_{S,0, \tau}^\varepsilon)} +  \|h_P^\varepsilon\|^2_{L^2(\Gamma_{P,0, \tau}^\varepsilon)} \big] + \frac 1{4\delta} \big[ \|K_S^\varepsilon(x, h_S^\varepsilon)\|^2_{L^2(S^\ve_\tau)}
\\& \qquad  + k_{r,ax}\| I_{\sqrt{\varepsilon}} K_P(h_P^\varepsilon)\|^2_{L^2(P^\ve_\tau)}
+ \|f(h_S^\ve)\|^2_{L^2(\Gamma_{S,0, \tau}^\varepsilon)} \!\! + \|\mathcal T_{\rm pot}\|^2_{L^2(\Gamma_{P,0, \tau}^\varepsilon)}\big],
\end{aligned}
\end{equation}
for $\tau \in (0, T]$ and  $0<\delta \leq \min\{K_{S,0}, \tilde K_{P,0}\}/2$, where $K_{S,0} = \min\{ K_{R,0}, K_{B,0}\}$ and $\tilde K_{P,0} =  k_{r,ax} K_{P,0}$, with $k_{r,ax}= \rho g  \min\{2\pi r k_{\rm r}, k_{\rm ax}/L_3\}$.
The definition of $\Theta_S^\ve$ and $\Theta_P$, implies
\begin{equation}\label{eqn4_lemma_2scale_convergence_1}
\begin{aligned}
\int_0^\tau \!\!\big\langle\partial_t {\theta}_S^\varepsilon(x, h_S^\varepsilon), h_S^\varepsilon \big\rangle_{V(S^\ve)^\prime} dt \geq & \int_{S^\varepsilon}\!\! {\Theta}_S^\varepsilon(x, h_S^\varepsilon(\tau))dx - \int_{S^\varepsilon}\!\! {\Theta}_S^\varepsilon(x, h_{S,0})dx,  \qquad  \\
\int_0^\tau \!\! \big\langle\partial_t {\theta}_P(h_P^\varepsilon), h_P^\varepsilon-a\big\rangle_{V(P^\ve)^\prime} dt & \geq
\int_{P^\varepsilon} \!\! {\Theta}_P(h_P^\varepsilon(\tau))dx - \int_{P^\varepsilon}\!\! {\Theta}_P(h_{P,0})dx\\
& -
a\int_{P^\varepsilon} \!\! {\theta}_P(h_P^\varepsilon(\tau))dx  + a\int_{P^\varepsilon} \!\!{\theta}_P(h_{P,0})dx ,
\end{aligned}
\end{equation}
for $\tau \in (0,T]$.
To estimate the third integral on the right hand-side of the second inequality in~\eqref{eqn4_lemma_2scale_convergence_1} we can use the continuity of $\theta_P$ and estimate
$$
| {\theta}_P(h_P^\varepsilon)| \leq \delta {\Theta}_P(h_P^\varepsilon) + \sup\limits_{|\sigma|\leq 1/2} |{\theta}_P(\sigma)|,
$$
which follows from the definition of ${\Theta}_P$ or, as in our case, use the boundedness of  $\theta_P$. The assumption on the microscopic structure of $S^\ve$ ensures existence of an extension~$\hat h_S^\varepsilon \in L^2(0,T;V(\Omega))$ of~$h_S^\varepsilon$ from~$(0,T)\times S^\varepsilon$ to~$(0,T)\times\Omega$  such that
\begin{equation}\label{eqn16_lemma_2scale_convergence_1}
\|\hat h_S^\varepsilon \|^2_{L^2(\Omega_T)} \leq C\| h_S^\varepsilon\|^2_{L^2(S^\ve_T)} , \quad
\| \nabla\hat h_S^\varepsilon\|^2_{L^2(\Omega_T)} \leq C\| \nabla h_S^\varepsilon \|^2_{L^2(S^\ve_T)},
\end{equation}
where $C>$ is independent of $\ve$, see e.g.~\citep[Theorem 2.10]{cioranescu2012homogenization}. By the trace theorem and~\eqref{eqn16_lemma_2scale_convergence_1}, we have
\begin{equation}\label{eqn17_lemma_2scale_convergence_1}
\begin{aligned}
\|h_S^\varepsilon\|^2_{L^2((0,T)\times \Gamma_{S,0}^\ve)}  \leq \| \hat h_S^\varepsilon \|^2_{L^2((0,T)\times \Gamma_0)}
\leq C\big[\|\hat h_S^\varepsilon\|^2_{L^2(\Omega_T)}  + \|\nabla \hat h_S^\varepsilon\|^2_{L^2(\Omega_T)}\big] \\\leq
C\big[\|h_S^\varepsilon\|^2_{L^2(S^\ve_T)}  + \|\nabla h_S^\varepsilon\|^2_{L^2(S^\ve_T)}\big].
\end{aligned}
\end{equation}
Applying the trace theorem over the unit cell $P =Y_P\times (-L_3,0)$  and the $\ve$-scaling in $(y_1,y_2)$-variables, together with the definition of the domain $P^\ve$, yields
\begin{equation}\label{eqn19_lemma_2scale_convergence_1}
\| h_P^\varepsilon\|^2_{L^2((0,T)\times \Gamma^\ve_{P,0})} \leq
C\big(\| h_P^\varepsilon\|^2_{L^2((0,T)\times P^\ve)} + \| I_\varepsilon\nabla h_P^\varepsilon\|^2_{L^2((0,T)\times P^\ve)}\big).
\end{equation}
Thus,  using \eqref{eqn4_lemma_2scale_convergence_1}, \eqref{eqn17_lemma_2scale_convergence_1}, and \eqref{eqn19_lemma_2scale_convergence_1} in inequality  \eqref{eqn1_lemma_2scale_convergence_1} implies
\begin{equation}\label{eqn20_lemma_2scale_convergence_1}
\begin{aligned}
\big\| {\Theta}_S^\varepsilon\big(\cdot, h_S^\varepsilon(\tau)\big)\big\|_{L^1(S^\ve)}
& + \big \| {\Theta}_P\big(h_P^\varepsilon(\tau)\big) \big \|_{L^1(P^\ve)} + \|\nabla h_S^\varepsilon\|^2_{L^2(S^\ve_\tau)}\\
&  	+ \| I_{\sqrt{\varepsilon}}\nabla h_P^\varepsilon\|^2_{L^2(P^\ve_\tau)}
\leq  \delta \Big[\| h_P^\varepsilon\|_{L^2(P^\ve_\tau)}^2  + \|h_S^\varepsilon\|_{L^2(S^\ve_\tau)}^2 \Big] + C_\delta,
\end{aligned}
\end{equation}
for $\tau \in (0,T]$ and   $0<\delta \leq \min\{K_{S,0}, \tilde K_{P,0}\}/(4C)$.
Applying  the generalised Poincar\'e inequality	   over~$P$ and using the boundary condition on~$\Gamma_{P,L_3}$ and the standard scaling argument,  for the first term on the right hand side of~\eqref{eqn20_lemma_2scale_convergence_1} we have
\begin{equation}\label{eqn21_lemma_2scale_convergence_1}
\begin{aligned}
\| h_P^\varepsilon\|^2_{L^2(P^\ve_\tau)} \leq C\big[\| I_\varepsilon\nabla h_P^\varepsilon\|^2_{L^2( P^\ve_\tau)} + T |a|^2 |\Gamma^\ve_{P,L_3}|\big],
\end{aligned}
\end{equation}
where~$C>0$ is independent of~$\varepsilon$. Similarly  for the second term on the right hand side of~\eqref{eqn20_lemma_2scale_convergence_1}, the generalised Poincar\'e inequality, the extension properties, and  $h_S^\varepsilon=0$ at $\Gamma^\ve_{S,L_3}$, imply
\begin{equation}\label{eqn23_lemma_2scale_convergence_1}
\begin{aligned}
\| h_S^\varepsilon\|_{L^2(S^\ve_\tau)} \leq \|\hat{h}_S^\varepsilon \|_{L^2(\Omega_\tau)} \leq C_1 \| \nabla\hat{h}_S^\varepsilon \|_{L^2(\Omega_\tau)} \leq C_2 \|\nabla h_S^\varepsilon\|_{L^2(S^\ve_\tau)},
\end{aligned}
\end{equation}
where~$C_1, C_2>0$ are independent of~$\varepsilon$.
Using now \eqref{eqn21_lemma_2scale_convergence_1} and \eqref{eqn23_lemma_2scale_convergence_1} in \eqref{eqn20_lemma_2scale_convergence_1}, and considering an appropriate $\delta$, yields the first  and the last estimates stated in the lemma.

The trace theorem over~$Y$, together with the standard scaling argument,   implies
\begin{equation}\label{eqn15_lemma_2scale_convergence_1}
\ve  \|h_J^\varepsilon\|^2_{L^2(\Gamma_{P,T}^\ve)} \leq C \big(\| h_J^\varepsilon\|^2_{L^2( J_T^\ve)}
+  \ve^2 \|\nabla_{\hat x} h_J^\varepsilon\|^2_{L^2(J_T^\ve)}\big), \quad \text{ for } J=S,P,
\end{equation}
with constant $C>0$ independent of~$\varepsilon$.
This, together with the  first estimate in \eqref{apriori_1_1},   yields  the estimates for the $L^2((0,T)\times \Gamma_P^\ve)$-norm.  Estimates \eqref{eqn17_lemma_2scale_convergence_1} and \eqref{eqn19_lemma_2scale_convergence_1} ensure the uniform in~$\ve$ boundedness of $h_S^\ve$ in $L^2((0,T)\times \Gamma_{S,0}^\ve)$ and $h_P^\ve$ in $L^2((0,T)\times \Gamma_{P,0}^\ve)$ respectively.
\end{proof}

Estimates in \eqref{apriori_1_1}, together with the properties of the extension \eqref{eqn16_lemma_2scale_convergence_1} and  of the two-scale convergence, see e.g.~\cite{Allaire_1992,   Nguetseng_1989, Neuss-Radu_1996}, imply the following convergence results
\begin{lemma} \label{lem:conver_hS}
There exist $h_S\in L^2(0,T;V(\Omega))$ and $h_{S,1} \in L^2((0,T)\times \Omega; H^1_{\rm per} (Y))$ such that
\begin{equation}\label{eqn_two_scale_convergence_weak_convergence_soil_extension}
\begin{aligned}
& \hat h_S^\varepsilon \rightharpoonup h_S && \text{ weakly in }  L^2(0,T;V(\Omega)), \\
& \hat h_S^\varepsilon \rightharpoonup h_S, \;\;  \nabla \hat h_S^\varepsilon \rightharpoonup \nabla h_S + \nabla_{y,0} h_{S,1}    && \text{ two-scale},  \\
& h_S^\ve  \rightharpoonup h_S   && \text{ two-scale  on }   (0,T)\times \Gamma_P^\ve,
\end{aligned}
\end{equation}
as $\ve \to 0$, where $\nabla_{y,0} u = (\partial_{y_1} u, \partial_{y_2} u, 0)^T$.
\end{lemma}

The next lemma provides the convergence result for the sequence~$\{h_P^\varepsilon\}$.
\begin{lemma}\label{lemma:two-scale_convergence_2}
For an extension~$\widetilde{h_P^\varepsilon}$ of~$h_P^\varepsilon$  by zero from~$(0,T)\times P^\varepsilon$ into~$(0,T)\times\Omega$, there exists $h_P\in L^2((0,T)\times\Omega)$, with $\partial_{x_3} h_P \in L^2((0,T)\times \Omega)$, such that, as $\ve \to 0$,
\begin{equation}\label{conver_hp_1}
\begin{aligned}
& \widetilde{h_P^\varepsilon}\rightharpoonup h_P \, \chi_{Y_P} && \text{ two-scale}, \\
& h_P^\ve  \rightharpoonup h_P  && \text{ two-scale on } \;  (0,T)\times \Gamma_P^\ve, \\
& \widetilde{h_P^\varepsilon}\rightharpoonup\frac{|Y_P|}{\lvert Y\rvert}h_P, \quad
\widetilde{\partial_{x_3}h_P^\varepsilon}\rightharpoonup\frac{|Y_P|}{\lvert Y\rvert}\partial_{x_3}h_P && \text{ weakly in } \;  L^2((0,T)\times\Omega),\\
& \varepsilon\widetilde{\nabla_{\hat{x}}h_P^\varepsilon}\rightharpoonup 0 && \text{ weakly in } \; L^2((0,T)\times\Omega)^2.
\end{aligned}
\end{equation}
\end{lemma}
\begin{proof}
From Lemma~\ref{lemma:two-scale_convergence_1} we have that the sequence~$\widetilde{h_P^\varepsilon}$ is bounded in~$L^2((0,T)\times \Omega)$ and, hence,  converges two-scale to~$h_P\in L^2((0,T)\times\Omega\times Y)$ with $h_P(t,x,y)=0$ for~$y\in Y\setminus\overline Y_P$, see e.g.~\cite{Allaire_1992}. From estimates in Lemma~\ref{lemma:two-scale_convergence_1} we also  have that~$\varepsilon^{1/2}\widetilde{\nabla_{\hat x} h_P^\varepsilon}$  converges two-scale and weakly in~$L^2((0,T)\times \Omega)^2$, and
\begin{equation}\label{eqn1_lemma_2scale_convergence_2}
\begin{aligned}
\varepsilon\int_0^T\!\!\!\!\int_{P^\varepsilon} \nabla_{\hat x} h_P^\varepsilon \, \psi\Big(t, x,\frac{\hat x}{\varepsilon}\Big) dxdt= \varepsilon\int_0^T\!\!\!\!\int_{\Omega} \widetilde{\nabla_{\hat x} h_P^\varepsilon} \, \psi\Big(t, x,\frac{\hat x}{\varepsilon}\Big) dxdt\to 0~\text{as }\varepsilon\to 0,
\end{aligned}
\end{equation}
for $\psi\in L^2(0,T;C_0^\infty(\Omega\times Y_P))^2$.  This, together with the two-scale convergence  of $\widetilde{h_P^\varepsilon}$  and with
\begin{equation}\label{eqn2_lemma_2scale_convergence_2}
\begin{aligned}
&	\varepsilon\int_0^T\!\!\!\!\int_{P^\varepsilon} \nabla_{\hat x} h_P^\varepsilon \, \psi\Big(t, x,\frac{\hat x}{\varepsilon}\Big) dxdt
= - \int_0^T\!\!\!\!\int_{P^\varepsilon}  h_P^\varepsilon \,\big[\ve {\rm div}_{\hat x}  \psi + {\rm div}_y  \psi \big]dxdt\\
&	= - \int_0^T\!\!\!\!\int_{\Omega}\widetilde{h_P^\varepsilon}\big[ \ve {\rm div}_{\hat x} \psi + {\rm div}_y \psi \big]dxdt
\to - \frac{1}{\lvert Y\rvert}\int_0^T\!\!\!\!\int_\Omega\int_{Y_P}h_P \, {\rm div}_y \psi \,  dydxdt,
\end{aligned}
\end{equation}
as $\ve \to 0$, implies~$h_P(t,x,y)=h_P(t,x)$ for $(t,x, y) \in (0,T)\times \Omega \times Y_P$. Choosing $\psi(t,x,y)=\psi(t,x)$, with $\psi \in C_0((0,T)\times \Omega)$, in the definition of the two-scale convergence of~$\widetilde{h_P^\varepsilon}$ gives
\begin{equation*}
\begin{aligned}
\lim_{\varepsilon\to0}\int_0^T\!\!\!\!\int_{\Omega}\widetilde{h_P^\varepsilon}(t,x)\psi(t, x)dxdt =
\int_0^T\!\!\!\!\int_{\Omega}\frac{|Y_P|}{\lvert Y\rvert}h_P(t,x)\psi(t,x)dxdt
\end{aligned}
\end{equation*}
and hence  the third convergence in~\eqref{conver_hp_1}.
The microscopic structure of $P^\ve$  implies $ \widetilde{\partial_{x_3}h_P^\varepsilon} = \partial_{x_3}\widetilde{h_P^\varepsilon}$. Then the estimate for  $\partial_{x_3}h_P^\varepsilon$, see Lemma~\ref{lemma:two-scale_convergence_1}, ensures  weak convergence of~$\widetilde{\partial_{x_3}h_P^\varepsilon}$  in~$L^2((0,T)\times\Omega)$ and, using the weak convergence of $\widetilde{h^\ve_P}$, we obtain the fourth  convergence in~\eqref{conver_hp_1}.
A priori estimates in \eqref{apriori_1_1} and properties of the two-scale convergence on oscillating boundaries, see e.g.~\cite{Allaire_1996, Neuss-Radu_1996}, ensure the second convergence result in~\eqref{conver_hp_1}. The uniform in $\ve$ estimate for $\varepsilon^\frac{1}{2}\widetilde{\nabla_{\hat{x}}h_P^\varepsilon}$, see~\eqref{apriori_1_1},  ensures  weak convergence of $\varepsilon^\frac{1}{2}\widetilde{\nabla_{\hat{x}}h_P^\varepsilon}$ in~$L^2((0,T)\times\Omega)^2$, and hence  the last convergence result in~\eqref{conver_hp_1}.
\end{proof}

To pass to the limit in the nonlinear terms in~\eqref{weak_formulation_soil} and \eqref{weak_formulation_plant}, we prove the strong two-scale convergence of ${\theta}_S^\varepsilon(\cdot, h_S^\varepsilon)$  and  ${\theta}_{P}(h_P^\varepsilon)$, by  showing the equicontinuity of the corresponding sequences and using  the Aubin-Lions-Simon compactness lemma~\cite{simon1986compact}.
\begin{lemma}\label{lemma:two_scale_convergence_3}
Under Assumption~\ref{assumption}, for  solutions of \eqref{richards_soil}-\eqref{bcs_root}  we have
\begin{equation}\label{equicont_theta_s}
\begin{aligned}
\int_0^{T-\lambda} \big \langle  \theta_S^\varepsilon(\cdot, h_S^\varepsilon(t+\lambda))- \theta^\ve_{S}(\cdot, h_S^\varepsilon(t)), h_S^\varepsilon(t+\lambda) - h_S^\varepsilon(t)\big\rangle_{S^\ve} dt \leq  C\lambda, \\
\int_0^{T-\lambda} \big\langle  {\theta}_P(h_P^\varepsilon(t+\lambda))- {\theta}_{P}(h_P^\varepsilon(t)), h_P^\varepsilon(t+\lambda) - h_P^\varepsilon(t)\big\rangle_{P^\ve}dt \leq  C\lambda,
\end{aligned}
\end{equation}
where $0<\lambda<T$ and   $C>0$ is independent of $\ve$.
\end{lemma}
\begin{proof}
Considering as a test function in~\eqref{weak_formulation_soil} the following function
$$ \psi^\varepsilon_\lambda(h_S^\ve):= \frac{1}{\lambda}\int_{t-\lambda}^t \!\!\big(h_S^\varepsilon(s+\lambda)-h_S^\varepsilon(s)\big)\chi_{(0,T-\lambda]}(s) ds,
$$
using  an integration by parts, and applying the Cauchy-Schwarz inequality  yield
\begin{equation*} 
\begin{aligned}
&\big \langle  {\theta}_{S}^\varepsilon(\cdot, h_S^\varepsilon(\cdot+\lambda))- {\theta}_{S}^\varepsilon(\cdot, h_S^\varepsilon), h_S^\varepsilon(\cdot+\lambda)-h_S^\varepsilon\big \rangle_{S^\ve_{T-\lambda}}= \lambda \big\langle\partial_t {\theta}_{S}^\varepsilon(\cdot, h_S^\varepsilon),\psi^\varepsilon_\lambda\big\rangle_{V^\prime(S^\ve), T-\lambda} \\
&  \leq
\lambda \big[\|K_S^\varepsilon(\cdot, h_S^\varepsilon)\nabla h_S^\varepsilon\|_{L^2( S^\varepsilon_{T-\lambda})} + \|K_S^\varepsilon(\cdot, h_S^\varepsilon)\|_{L^2(S^\varepsilon_{T-\lambda})} \big] \|\nabla \psi^\varepsilon_\lambda\|_{L^2( S^\varepsilon_{T-\lambda})}
\\
&\qquad  +  \lambda  \ve  k_\Gamma\big[\|h_S^\varepsilon\|_{L^2(\Gamma_{P, T-\lambda}^\varepsilon)}
 +\|h_P^\varepsilon\|_{L^2( \Gamma_{P, T-\lambda}^\varepsilon)}\big]
\|\psi^\varepsilon_\lambda\|_{L^2(\Gamma_{P, T-\lambda}^\varepsilon)} \\
& \qquad + \lambda\|f(h_S^\varepsilon)\|_{L^2((0,T-\lambda)\times \Gamma_{S,0}^\varepsilon)} \|\psi^\varepsilon_\lambda\|_{L^2((0,T-\lambda)\times \Gamma_{S,0}^\varepsilon)}.
\end{aligned}
\end{equation*}
A priori estimates in  Lemma~\ref{lemma:two-scale_convergence_1} and   the uniform boundedness of $K_S^\varepsilon$ and $f$ imply the first estimate in~\eqref{equicont_theta_s}. Analogous calculations, with $\psi_\lambda^\ve(h_P^\varepsilon)$  as a test function in~\eqref{weak_formulation_plant}, yield the second estimate in~\eqref{equicont_theta_s}.
\end{proof}

To show the strong two-scale convergence of ${\theta}_S^\varepsilon(x, h_S^\varepsilon)$ we use the  unfolding operator
$\mathcal{T}^\varepsilon:L^p((0,T)\times S^\varepsilon)\to L^p((0,T)\times\Omega\times \hat S)$, for $1<p<\infty$, given by
$$
\mathcal{T}^\varepsilon(h_S^\varepsilon)(t,x,y)=h_S^\varepsilon\Big(t,\varepsilon\Big[\frac{\hat{x}}{\varepsilon}\Big]_{Y}+\varepsilon y, x_3\Big) \qquad \text{ for } \; t \in (0,T), \; x \in \Omega, \; y \in \hat S,
$$
where     $\hat S = Y \setminus \overline Y_P$, $\hat{x}=(x_1,x_2)$,~$y=(y_1,y_2)$,  and~$[\frac{\hat{x}}{\varepsilon}]_{Y}$ is the unique integer combination, see e.g.~\citep{cioranescu2018periodic}. We have a similar definition  for $\mathcal{T}^\varepsilon:L^p((0,T)\times P^\varepsilon)\to L^p((0,T)\times\Omega\times Y_P)$.

\begin{lemma}\label{lemma:two-scale_convergence_theta_S}
Under Assumption~\ref{assumption}, for  solutions of \eqref{richards_soil},~\eqref{bcs_soil}  we have
$$
\begin{aligned}
 & \mathcal{T}^\varepsilon(h_S^\varepsilon) \to h_S && \text{ a.e.~in }  (0,T)\times \Omega \times \hat S, \\
 & {\theta}_S^\varepsilon(x, h_S^\varepsilon)\to {\theta}_S(y,h_S)=\chi_{\hat R}(y) {\theta}_{R}(h_S) + \chi_{\hat B}(y) {\theta}_{B}(h_S) \quad && \text{strongly two-scale}.
\end{aligned}
$$
\end{lemma}
\begin{proof}
First we show the strong convergence of  $ {\theta}_{R}(\mathcal{T}^\varepsilon(h_S^\varepsilon))$  in~$L^2((0,T)\times\Omega\times \hat R)$ and of ${\theta}_{B}(\mathcal{T}^\varepsilon(h_S^\varepsilon))$  in~$L^2((0,T)\times\Omega\times \hat B)$. The Lipschitz continuity of $\theta_J$, with $J=R,B$,  properties of the unfolding operator, see e.g.~\cite{cioranescu2018periodic}, and Lemma~\ref{lemma:two_scale_convergence_3} imply
\begin{equation*}\label{relative_compact_x}
\begin{aligned}
&\Big\|\int_{t_1}^{t_2} \!\! {\theta}_J\big(\mathcal{T}^\varepsilon(h_S^\varepsilon)(t, \cdot+r, \cdot)\big) dt  - \int_{t_1}^{t_2} \!\! {\theta}_J\big(\mathcal{T}^\varepsilon(h_S^\varepsilon)\big)dt\Big\|_{L^2(\Omega\times \hat J)}^2\\
&  \leq  C_1\big\|{\theta}_J\big(h_S^\varepsilon(\cdot, \cdot +r)\big) - {\theta}_J\big(h_S^\varepsilon\big)\big\|^2_{L^2((t_1, t_2)\times J^\ve)}
\leq C_2|r|^2 \|\nabla{h_S^\varepsilon}\|_{L^2((t_1, t_2)\times J^\varepsilon)}^2 \leq  C |r|^2,\\
&\Big\|\int_{t_1}^{t_2}\!\! {\theta}_J\big(\mathcal{T}^\varepsilon(h_S^\varepsilon)(t, \cdot, \cdot+r_1)\big)dt - \int_{t_1}^{t_2}\!\! {\theta}_J\big(\mathcal{T}^\varepsilon(h_S^\varepsilon)\big)dt\Big\|_{L^2(\Omega\times \hat J)}^2\\
&    \leq C_1 |r_1|^2 \| \nabla_y \mathcal{T}^\varepsilon\big( {\theta}_J(h_S^\varepsilon) \big) \big\|_{L^2((t_1, t_2)\times \Omega \times \hat J)}^2
\leq C_2 |r_1|^2 \ve^2 \|\nabla{h_S^\varepsilon}\|_{L^2((t_1, t_2)\times J^\varepsilon)}^2 \leq  C |r_1|^2,\\
& \big \| {\theta}_{J}(\mathcal{T}^\varepsilon(h_S^\varepsilon)(\cdot+\lambda, \cdot, \cdot))
- {\theta}_{J}(\mathcal{T}^\varepsilon(h_S^\varepsilon))\big\|^2_{L^2(\Omega_{T-\lambda} \times \hat J)}
\\ & \quad  \leq C_1\int_0^{T-\lambda}\!\!\!\big\langle  {\theta}_{S}^\varepsilon(\cdot, h_S^\varepsilon(t+\lambda))- {\theta}_{S}^\varepsilon(\cdot, h_S^\varepsilon(t)), h_S^\varepsilon(t + \lambda)-h_S^\varepsilon(t)\big \rangle_{S^\ve} dt \leq  C\lambda,
\end{aligned}
\end{equation*}
for~$r\in\mathbb{R}^3$, $r_1 \in \mathbb R^2$, and $\lambda > 0$, where $J=R,B$ and  $C_1, C_2, C >0$ are independent of $\ve$.
Combining the estimates from above and using the compactness result in~\citep{simon1986compact}, yield the existence of $z_R \in L^2((0,T)\times\Omega\times \hat R)$ and $z_B\in L^2((0,T)\times\Omega\times \hat B)$  such that
\begin{equation}\label{two-scale:strong_convergence_something_R}
{\theta}_{J}\big(\mathcal{T}^\varepsilon(h_S^\varepsilon)\big)\to z_J\quad\text{strongly in } \;\; L^2((0,T)\times\Omega\times \hat J), \quad \text{ for } \; \;  J=R,B.
\end{equation}
Since~$ {\theta}_{B}$ and~$ {\theta}_{R}$ are strictly increasing and continuous, we have
\begin{equation}\label{two-scale:unfolded_a_e_convergence_RB}
\mathcal{T}^\varepsilon(h_S^\varepsilon)=  {\theta}_{J}^{-1}\big( {\theta}_{J}\big(\mathcal{T}^\varepsilon(h_S^\varepsilon)\big)\big)\to {\theta}_{J}^{-1}(z_J)\quad\text{a.e.~in }(0,T)\times\Omega\times \hat J,
\end{equation}
for $ J=R,B$. The two-scale convergence of $h_S^\ve$ implies $\mathcal{T}^\varepsilon(h_S^\varepsilon)\rightharpoonup h_S$ in $L^2((0,T)\times\Omega\times \hat S)$, see e.g.~\cite{cioranescu2018periodic}.
Using~\eqref{two-scale:strong_convergence_something_R}  and~\eqref{two-scale:unfolded_a_e_convergence_RB}, together with $\mathcal{T}^\varepsilon( {\theta}_{J}(h_S^\varepsilon))=   {\theta}_{J}(\mathcal{T}^\varepsilon(h_S^\varepsilon))$,  ensures
$\mathcal{T}^\varepsilon\big( {\theta}_{J}(h_S^\varepsilon)\big) \to  {\theta}_{J}(h_S)$ strongly in  $L^2\big((0,T)\times\Omega\times \hat J\big)$
and  ${\theta}_{J}(h_S^\varepsilon)\to {\theta}_{J}(h_S)$ strongly two-scale,  for $J=R,B$. Then  ${\theta}_S^\varepsilon(x,h_S^\varepsilon) = \chi_{\hat R}(\hat x/\ve) {\theta}_{R}(h_S^\varepsilon) + \chi_{\hat B}(\hat x/\ve) {\theta}_{B}(h_S^\varepsilon)$ converges strongly two-scale to ${\theta}_S(\cdot,h_S) \in L^2((0,T)\times \Omega\times \hat S)$.
\end{proof}	
\begin{lemma}\label{lemma:strong_conv_hP}
Under Assumption~\ref{assumption}, for  solutions of \eqref{richards_root},~\eqref{bcs_root}  we have
\begin{equation}\label{twoscale_conv_hp}
\begin{aligned}
& \mathcal T^\ve (h^\ve_P) \to h_P && \text{ a.e.~in } \, (0,T)\times \Omega\times Y_P, \\
& {\theta}_P(h_P^\varepsilon) \to {\theta}_P(h_P)\chi_{Y_P} && \text{ strongly  two-scale}.
\end{aligned}
\end{equation}
\end{lemma}

\begin{proof}
To prove strong convergence of ${\theta}_P(h_P^\varepsilon)$ we first  show the equicontinuity of
$\mT^\ve(\theta_P(h^\ve_P))$ by using  similar arguments as in the proof of the uniqueness result in Theorem~\ref{theorem_microscopic_models_uniqueness}.
Consider $\hat \Omega = \Omega \cap \{ x_3 =0\}$ and
$\Omega_\varsigma =\{ x \in \Omega : {\rm dist} ( \hat x, \partial \hat \Omega ) \geq \varsigma\}$, with $\varsigma>0$ and $\hat x = (x_1, x_2)$, and define
$$
\begin{aligned}
\hat P^\ve_\varsigma =  \bigcup\limits_{\xi \in \Xi_\varsigma^\ve} \ve(Y_P + \xi), \quad  \hat \Gamma^\ve_{P_\varsigma} =  \bigcup\limits_{\xi \in \Xi_\varsigma^\ve} \ve(\Gamma_P + \xi),\;
\end{aligned}
$$
and $P^\ve_\varsigma = \hat P^\ve_\varsigma \times (-L_3, 0)$, $\Gamma^\ve_\varsigma = \hat \Gamma^\ve_{P_\varsigma} \times (-L_3, 0)$,
where $\Xi^\ve_\varsigma = \{ \xi \in \Xi^\ve :  \ve(Y + \xi)\times (-L_3, 0) \subset \Omega_\varsigma\}$.
For  $l\in\mathbb{Z}^2$, such that~$|l\varepsilon|<\varsigma$, and  $\psi^{-l}(t,x) = \psi(t,\hat{x}-l\varepsilon, x_3)$, where $\psi\in C^1_0(0,T;V(P_\varsigma^\varepsilon))$, we have  $\psi^{-l}\in C^1_0(0,T;V(P_{\varsigma, l}^\varepsilon))$, with
$P_{\varsigma, l}^\varepsilon=\{x+(\varepsilon l_1, \varepsilon l_2, 0)^T: x\in P_\varsigma^\varepsilon\}$.
Using $\psi^{-l}$ in the weak formulation of~\eqref{richards_root} and \eqref{bcs_root} over $P_{\varsigma, l}^\varepsilon$, integrating  by parts in the time derivative  and changing variables from~$x$ to~$x-(\varepsilon l_1, \varepsilon l_2, 0)^\top$, for  $h_P^{\ve,l}(t,x) = h_P^\ve(t, \hat x+ \ve l, x_3)$, we obtain
\begin{equation}\label{l_plant_variational_problem}
\begin{aligned}
\langle\partial_t{\theta}_P(h_P^{\varepsilon, l}), \psi\rangle_{V(P^\ve_\varsigma)^\prime, T} +
\langle H_r^\varepsilon K_P(h_P^{\varepsilon, l})(\nabla h_P^{\varepsilon, l} + e_3), \nabla \psi \rangle_{P^\ve_{\varsigma, T}}
\\ + \varepsilon k_\Gamma \langle h_P^{\varepsilon, l} - h_S^{\varepsilon, l}, \psi \rangle_{\Gamma^\ve_{P_\varsigma, T}} +  \langle \mathcal{T}_\text{pot}, \psi \rangle_{(0,T)\times\Gamma^\ve_{P_\varsigma, 0}}  & =0.
\end{aligned}
\end{equation}
 Consider now the functions~$\sigma_\delta^+$, $\sigma_\delta^-$,
$\eta_{P,\delta}^+(h_P^{\varepsilon, l},w^0)$, and~$\eta_{P,\delta}^-(h_P^\varepsilon,w^0)$, with $w^0-a\in V(P_\varsigma^\varepsilon)$, as in~\eqref{our_eta_prime}.
The arguments similar to those in the proof of Theorem~\ref{theorem_microscopic_models_uniqueness} and Lemma~\ref{lemma_uniqueness_lemma_1} yield that  $\zeta^+_\tau = \int_t^{t+\tau}\!\!\sigma_\delta^+(h_P^{\varepsilon, l} - w^0)\kappa(s) ds$ and $\zeta^-_\tau = \int_t^{t+\tau}\!\!\sigma_\delta^-(h_P^{\varepsilon} - w^0)\kappa(s) ds$,  for $\tau >0$ and $\kappa\in C_0^\infty(0,T)$,
are admissible test functions in~\eqref{l_plant_variational_problem} and in the weak formulation of \eqref{richards_root} over~$P_\varsigma^\varepsilon$ respectively, and we obtain the following inequalities
\begin{equation}\label{var_form_eps_l}
\begin{aligned}
- \big\langle \eta_{P,\delta}^+(h_P^{\varepsilon, l},w^0),\frac{d\kappa}{dt}\big\rangle_{P_{\varsigma, T}^\varepsilon}
+ \big\langle H_r^\varepsilon K_P(h_P^{\varepsilon, l})(\nabla h_{P}^{\varepsilon, l} + e_3),
\nabla \sigma_\delta^+(h_P^{\varepsilon, l} - w^0)\kappa\big\rangle_{P^\ve_{\varsigma, T}}  \\
 +
\varepsilon \, k_\Gamma \big\langle h_P^{\varepsilon, l}-h_S^{\varepsilon, l}, \sigma_\delta^+(h_P^{\varepsilon, l}- w^0)\kappa\big \rangle_{\Gamma^\ve_{P_\varsigma, T}}  + \big\langle  \mathcal{T}_\text{pot}, \sigma_\delta^+(h_P^{\varepsilon, l}- w^0)\kappa \big\rangle_{\Gamma_{P_\varsigma, 0, T}^\ve}   \leq 0, \\
- \big\langle \eta_{P,\delta}^-(h_P^\varepsilon,w^0), \frac{d\kappa}{dt}\big\rangle_{P_{\varsigma, T}^\varepsilon}
+
\big\langle H_r^\varepsilon K_P(h_P^\varepsilon)(\nabla h_P^{\varepsilon} + e_3), \nabla \sigma_\delta^-(h_P^{\varepsilon}- w^0)\kappa \big\rangle_{P^\ve_{\varsigma, T}}\\
+ \varepsilon \, k_\Gamma \big\langle h_P^{\varepsilon}-h_S^{\varepsilon},\sigma_\delta^-(h_P^{\varepsilon}- w^0)\kappa \big\rangle_{\Gamma^\ve_{P_\varsigma, T}}
 + \big\langle \mathcal{T}_\text{pot}, \sigma_\delta^-(h_P^{\varepsilon}- w^0)\kappa \big \rangle_{\Gamma^\ve_{P_\varsigma, 0, T}}  \leq 0.
\end{aligned}
\end{equation}
Considering  a doubling of the time variable~$(t_1,t_2)\in(0,T)^2$,  with
$h_P^{\varepsilon,l}(x,t_1, t_2)=h_P^{\varepsilon,l}(x,t_1)$, $w^0 = h_P^\varepsilon(x,t_2)$, and $\kappa = \kappa(t_2):t_1\to \kappa(t_1,t_2)$
in the first inequality,  and $h_P^{\varepsilon}(x,t_1, t_2)=h_P^{\varepsilon}(x,t_2)$, $w^0 = h_P^{\varepsilon, l}(x,t_1)$, and
$\kappa = \kappa(t_1):t_2\to \kappa(t_1,t_2)$ in the second inequality in~\eqref{var_form_eps_l}, with non-negative $\kappa\in C_0^\infty((0,T)^2)$, and   adding the resulting inequalities yield
\begin{equation*}
\begin{aligned}
&\int_{(0,T)^2}\!\Big[-\int_{P_\varsigma^\varepsilon}\! \Big(\eta_{P,\delta}^+(h_P^{\varepsilon, l},h_P^\varepsilon)\partial_{t_1}\kappa + \eta_{P,\delta}^-(h_P^\varepsilon,h_P^{\varepsilon, l})\partial_{t_2}\kappa\Big) dx\\
			&\qquad +
\big\langle H_r^\varepsilon K_P(h_P^{\varepsilon, l})(\nabla h_P^{\varepsilon, l} + e_3) - H_r^\varepsilon K_P\big(h_P^\varepsilon\big)(\nabla h_P^\varepsilon + e_3), \nabla \sigma_\delta^+(h_P^{\varepsilon, l}- h_P^\varepsilon)\kappa \big\rangle_{P^\ve_{\varsigma}}\\
&\qquad + \varepsilon\,  k_\Gamma \big\langle (h_P^{\varepsilon, l}-h_P^\varepsilon)-(h_S^{\varepsilon, l}-h_S^\varepsilon), \sigma_\delta^+(h_P^{\varepsilon, l}- h_P^\varepsilon)\kappa \big\rangle_{\Gamma_{P_\varsigma}^\ve}\Big]dt_1dt_2\leq 0.
\end{aligned}
\end{equation*}
Taking~$\delta\to 0$ in the above inequality and applying the arguments similar to the one used in the proof of \eqref{uniqueness_inequality3_doubled_time}, 
give
\begin{eqnarray}\label{DCT_est_two_scale_h_P}
-\big\langle (\theta_P(h_P^{\varepsilon, l}) - \theta_P(h_P^\varepsilon))^+, \partial_{t_1}\kappa + \partial_{t_2}\kappa\big\rangle_{(0,T)^2 \times P^\ve_\varsigma}
+
\varepsilon  k_\Gamma \big\langle(h_P^{\varepsilon, l}-h_P^\varepsilon)^+, \kappa \big\rangle_{(0,T)^2 \times \Gamma^\ve_{P_\varsigma}} \nonumber
\\ \leq \varepsilon  k_\Gamma
\big\langle (h_S^{\varepsilon, l}-h_S^\varepsilon)^+\text{sign}^+(h_P^{\varepsilon, l}-h_P^{\varepsilon}), \kappa \big\rangle_{(0,T)^2 \times \Gamma^\ve_{P_\varsigma}}.
\end{eqnarray}
The  function $\kappa_\varrho$ defined as in \eqref{kappa_varrho}
is admissible in~\eqref{DCT_est_two_scale_h_P}, in place of~$\kappa$, and applying the change of variables~$\tau = t_1-t_2$ and denoting~$t=t_1$ we obtain
\begin{equation}\label{hp_2scale_inequality2_doubled_time_changed_variable}
\begin{aligned}
\int_{\mathbb{R}}\!\frac{1}{\varrho}\vartheta\big(\frac{\tau}{\varrho}\big)\!
\int_0^T\!\!\Big[&-\int_{P_\varsigma^\varepsilon} \!\!\big(\theta_P(h_P^{\varepsilon, l}(t)) - \theta_P(h_P^\varepsilon(t-\tau))\big)^+\partial_t\kappa\big(t-\frac \tau2\big)dx\\
&+ \varepsilon  k_\Gamma\int_{\Gamma_{P,\varsigma}^\varepsilon}\!\!\!(h_P^{\varepsilon, l}(t)-h_P^\varepsilon(t-\tau))^+\kappa\big(t - \frac \tau 2\big) d\gamma\Big] dtd\tau\\
\leq \varepsilon  k_\Gamma \!\int_{\mathbb{R}}\!\frac{1}{\varrho}\vartheta\big(\frac{\tau}{\varrho}\big)&\int_0^T \!\!\!\int_{\Gamma_{P,\varsigma}^\varepsilon}\!\!\!\!\! (h_S^{\varepsilon, l}(t)-h_S^\varepsilon(t-\tau))^+\text{sign}^+(h_P^{\varepsilon, l}-h_P^{\varepsilon})\kappa\big(t - \frac \tau 2\big) d\gamma dtd\tau.
\end{aligned}
\end{equation}
Taking~$\varrho\to0$, applying the integration by parts,
and using the compact support of~$\kappa$, imply
\begin{equation}\label{hp_2scale_inequality3_doubled_time_changed_variable}
\begin{aligned}
\int_{0}^T\!\! \! \!\!\kappa(t) \partial_t\!\!\int_{P_\varsigma^\varepsilon} \!\!\! \!\big(\theta_P(h_P^{\varepsilon, l}(t)) - \theta_P(h_P^\varepsilon(t))\big)^+\! dxdt  \leq  \varepsilon k_\Gamma  \!\! \int_{0}^T\!\! \! \!\! \kappa(t)\!\! \int_{\Gamma_{P,\varsigma}^\varepsilon} \!\!\!\! \!\!\!|h_S^{\varepsilon, l}(t)-h_S^\varepsilon(t)| d\gamma dt.
\end{aligned}
\end{equation}
Exchanging $h_P^{\varepsilon, l}$ and $h_P^{\varepsilon}$ in the calculations above yields~\eqref{hp_2scale_inequality3_doubled_time_changed_variable} for $(\theta_P(h_P^{\varepsilon}) - \theta_P(h_P^{\varepsilon, l}))^+$.
Applying the trace theorem over the unit cell~$Y$, together with the standard scaling argument,   and using estimates for $h^\ve_S$ in Lemma~\ref{lemma:two-scale_convergence_1}, yield
\begin{equation}\label{estim_hPl-hP}
\int_{0}^T\!\!\!\!\kappa(t) \, \partial_t\int_{P_\varsigma^\varepsilon} |\theta_P(h_P^{\varepsilon, l}(t,x)) - \theta_P(h_P^\varepsilon(t,x))| dxdt \leq  C \ve l
\end{equation}
for any non-negative $\kappa \in C^\infty_0(0,T)$. Using the regularity of  initial conditions and Lipschitz continuity of $\theta_P$, from \eqref{estim_hPl-hP} we obtain
$$
\|\theta_P(h_P^{\varepsilon, l}(t)) - \theta_P(h_P^\varepsilon(t))\|_{L^1(P^\ve_\varsigma)} \leq  C \ve l, \quad \text{ for } \;  t \in (0,T].
$$
Considering $|z|\leq \varsigma$ and using the definition of the unfolding operator  imply
 \begin{equation}\label{equicont}
 \begin{aligned}
 &\int_{\Omega_T}\! \int_{Y_P}\!\! |\theta_P(\mT^\ve (h^{\ve}_P)(t, \hat x +z, x_3,y)) - \theta_P(\mT^\ve (h^{\ve}_P)(t, x, y))|dy dx dt \\
 & \quad \leq \sum_{k \in \{0,1\}^2} \int_0^T\!\!\!\int_{P^\ve_\varsigma} \!\!|\theta_P(h^{\ve,l_k}_P(t,x)) - \theta_P(h^{\ve}_P(t,x))| dxdt \leq \kappa(\varsigma) \to 0
 \end{aligned}
 \end{equation}
as $\varsigma \to 0$, where $l_k = k+ [z/\ve]$ and $|\ve l_k| \leq \varsigma$ for all $\ve >0$, such that $\ve \leq \ve_0$ for some $\ve_0>0$.
For the finite number of $\ve > \ve_0$ estimate~\eqref{equicont} follows from the continuity of the~$L^2$-norm. The  estimates for  $\ve \nabla_{\hat x} h^\ve_P$ and $\partial_{x_3}  h^\ve_P$  ensure, for  $r \in \mathbb R^2$ and $r_1 \in \mathbb R$,
 \begin{equation}\label{equicont_3}
 \begin{aligned}
 &  \|\theta_P(\mT^\ve (h^{\ve}_P)(\cdot, \cdot, \cdot,  y+r)) - \theta_P(\mT^\ve (h^{\ve}_P))\|^2_{L^2(\Omega_T\times Y_P)}  \\
 & \qquad  \leq C |r|^2 \|\nabla_y \mT^\ve(\theta_P(h^{\ve}_P))\|^2_{L^2(\Omega_T \times Y_P)}
  \leq C|r|^2\ve^2 \|\nabla_{\hat x} h^{\ve}_P\|^2_{L^2(P^\ve_T)} \leq C |r|^2, \\
 & \|\theta_P(\mT^\ve (h^{\ve}_P)(\cdot, \cdot, x_3+ r_1,\cdot)) - \theta_P(\mT^\ve (h^{\ve}_P))\|^2_{L^2(\Omega_T\times Y_P)}  \\ & \qquad \leq   C r_1^2 \|\partial_{x_3} h^{\ve}_P\|^2_{L^2(P^\ve_T)} \leq C r_1^2.
 \end{aligned}
 \end{equation}
 Using  \eqref{equicont}, \eqref{equicont_3}, and the second estimate in~\eqref{equicont_theta_s}, and applying the compactness theorem, see~\cite{simon1986compact} and boundedness of $\theta_P$, yields the strong convergence of $\theta_P(\mT^\ve (h^{\ve}_P))$ in $L^2((0,T)\times\Omega \times Y_P)$.
 This, together with the monotonicity and continuity of $\theta_P$ and two-scale convergence of $h^\ve_P$ to $h_P \chi_{Y_P}$,  implies  $\mathcal T^\ve( h^\ve_P) \to  h_P$ a.e.~in $(0,T)\times \Omega \times Y_P$. Hence
 $\mT^\ve (\theta_P(h^{\ve}_P)) \to \theta_P(h_P)$ strongly in $L^2((0,T)\times \Omega \times Y_P)$, which, applying the  properties of the unfolding operator, implies the strong two-scale convergence of $\theta_P(h^{\ve}_P)$, stated in the lemma.
 \end{proof}

 \begin{theorem}\label{theorem: limit model}
  Under Assumption~\ref{assumption},  a sequence of solutions $(h_S^\ve, h_P^\ve)$ of microscopic model~\eqref{richards_soil}--\eqref{bcs_root}  converges to  solution $h_S \in L^2(0,T; V(\Omega))$,  $h_P-a \in L^2(0,T; U(\Omega))$ of the macroscopic problem
  \begin{equation} \label{macro_model}
   \begin{aligned}
    & \partial_t \theta_S^\ast(h_S) - \nabla \cdot(K_{S,\rm hom}(h_S) (\nabla h_S + e_3)) =  k_\Gamma \vartheta_\Gamma (h_P - h_S) && \text{in }   (0,T)\times \Omega,  \\
    &  \partial_t \theta_P(h_P) -  \partial_{x_3} (\tilde K_P(h_P)(\partial_{x_3} h_P + 1)) = k_\Gamma  \vartheta_{\Gamma,P} (h_S - h_P) && \text{in }   (0,T)\times \Omega, \\
    & K_{S, \rm hom}(h_S) (\nabla h_S + e_3) \cdot \nu = 0 \quad &&  \text{on }  (0,T)\times \Gamma_N, \\
    & K_{S, \rm hom}(h_S) (\nabla h_S + e_3) \cdot \nu = {\color{blue} -}\vartheta_S f(h_S) \quad &&  \text{on }  (0,T)\times \Gamma_0, \\
    &  \tilde K_P(h_P)(\partial_{x_3} h_P + 1) = {\color{blue} -}  \mathcal T_{\rm pot}   && \text{on } (0,T)\times  \Gamma_0, \\
    & h_S(0) = h_{S,0}, \qquad h_P(0) = h_{P,0} && \text{in }  \Omega,
   \end{aligned}
  \end{equation}
  where $\theta_S^\ast= \vartheta_R \theta_R(h_S) + \vartheta_B \theta_B(h_S)$, $\tilde K_P(h_P) = (k_{\rm ax} \rho g/ L_3) K_P(h_P)$,
  $\vartheta_J = |\hat J|/|Y|$, for $J=R,B, S$, $\vartheta_{\Gamma, P} = |\Gamma_P|/|Y_P|$,  $\vartheta_\Gamma = |\Gamma_P|/|Y|$,  and
  $\Gamma_N = \partial \hat \Omega \times (-L_3, 0)$,
  $$
  \begin{aligned}
  V(\Omega) &= \{ v \in H^1(\Omega) : v =0 \text{ on } \Gamma_{L_3} \},  \\
  U(\Omega) &= \{ w \in L^2(\Omega): \partial_{x_3} w\in L^2(\Omega),   w = 0 \text{ on } \Gamma_{L_3} \},
  \end{aligned}
  $$   and
  $$
  \begin{aligned}
 & K_{S, \rm hom, ij}(h_S) = \frac 1{|Y|} \int_{\hat S} \Big[ K_{S}(y, h_S) \delta_{ij} + K_{S}(y, h_S) \partial_{y_i} w^j \big] dy, \quad i,j=1,2, \\
 & K_{S, \rm hom, 3j}(h_S)= K_{S, \rm hom, j3}(h_S) = \frac 1{|Y|} \int_{\hat S}  K_{S}(y, h_S) dy \,\delta_{3 j},  \quad  j=1,2,3,
  \end{aligned}
  $$
  with $w^j$, for $j=1,2$, being the solutions of the unit cell problems
  \begin{equation}\label{unit_cell_prob_macro}
   \begin{aligned}
 \nabla_y \cdot \big( K_S(y, h_S)(\nabla_y w^j + e_j)\big) & = 0 \; \; \;  \text{ in } \hat S, \\
 K_S(y, h_S)(\nabla_y w^j + e_j)\cdot \nu & = 0 \; \; \; \text{ on } \Gamma_P, \; && w^j \;\;  Y - \text{ periodic},
   \end{aligned}
  \end{equation}
  where $\{ e_1, e_2\}$ is the  canonical basis in $\mathbb R^2$. \\
 Under additional assumptions that  $K_J$, for $J=B, R, P$, are Lipschitz continuous
  and $f$ is non-decreasing, the solution of macroscopic model~\eqref{macro_model} is unique.
 \end{theorem}

 \begin{proof}
 Considering $\phi(t,x) = \phi_1(t,x) + \ve \phi_2(t,x, \hat x/\ve)$, with $\phi_1 \in C^\infty_0(0,T; C^\infty_{\Gamma_{L_3}}\!\!(\Omega))$ and $\phi_2 \in C^\infty_0((0,T)\times \Omega; C^\infty_{\rm per}(Y))$,  and $\psi \in C^\infty_0(0,T;C^\infty_{\Gamma_{L_3}}(\Omega))$ as test functions in~\eqref{weak_formulation_soil} and \eqref{weak_formulation_plant}
 and applying the unfolding operator  we have
 \begin{equation}\label{unfold_conv_1_S}
\begin{aligned}
&	- \big\langle {\theta}_S(y,\mT^\ve( h_S^\ve)), \partial_t \mT^\ve(\phi_1)  + \ve   \partial_t \mT^\ve(\phi_2) \big\rangle_{\Omega_T\times \hat S}  \\
& \qquad +   k_\Gamma \big\langle \mT^\ve(h_S^\ve)- \mT^\ve(h_P^\ve),\mT^\ve(\phi_1) + \ve \mT^\ve(\phi_2) \big\rangle_{\Omega_T\times \Gamma_P}
	\\
&	\quad  + \big\langle K_S(y, \mT^\ve(h_{S}^\ve))(\mT^\ve(\nabla h_{S}^\ve) + {e}_3), \mT^\ve(\nabla \phi_1) + \ve \mT^\ve(\nabla \phi_2)\big \rangle_{\Omega_T \times \hat S}  \\
 & \qquad
 =  - \big \langle f(\mT^\ve(h_S^\ve)), \mT^\ve(\phi_1) + \ve \mT^\ve(\phi_2) \big\rangle_{\Gamma_{0, T}\times \hat S},
\end{aligned}
\end{equation}
and
\begin{equation}\label{unfold_conv_1_P}
\begin{aligned}
& - \big\langle {\theta}_P(\mT^\ve(h_P^\ve)), \partial_t \mT^\ve(\psi)\big\rangle_{\Omega_T\times Y_P}  	+  k_\Gamma \big \langle  \mT^\ve(h_P^\ve)- \mT^\ve(h_S^\ve), \mT^\ve(\psi) \big\rangle_{\Omega_T\times \Gamma_P}
\\
&	\qquad + \big\langle H_r^\ve K_P(\mT^\ve(h_{P}^\ve))(\mT^\ve(\nabla h_{P}^\ve) + {e}_3),\mT^\ve( \nabla \psi)\big \rangle_{\Omega_T\times Y_P}
\\
& \qquad = - \big\langle \mathcal{T}_\text{pot},\mT^\ve(\psi) \big\rangle_{\Gamma_{0,T} \times Y_P}.
\end{aligned}
\end{equation}
Using the properties of the unfolding operator $\mT^\ve$, i.e.~$\mT^\ve(\phi) \to \phi$ for $\phi \in L^p((0,T)\times A)$, with $1<p<\infty$ and $A = \Omega$  or $A= \partial \Omega$, $\mT^\ve(\phi (\cdot, \cdot , \cdot/\ve)) \to \phi$ for $\phi \in L^p((0,T)\times A; C_{\rm per} (Y))$,   and $\ve \mT^\ve(\nabla \phi(\cdot,\cdot,\cdot/\ve)) \to \nabla_y \phi$ for $\phi \in L^p(0,T; W^{1,p}( \Omega); C^1_{\rm per}(Y))$,  and the relations between the two-scale (strong two-scale) convergence of a sequence and  weak (strong) convergence of the corresponding unfolded sequence, see e.g.~\cite{cioranescu2018periodic}, together with the convergence results in Lemmas~\ref{lem:conver_hS},~\ref{lemma:two-scale_convergence_2},~\ref{lemma:two-scale_convergence_theta_S}, and~\ref{lemma:strong_conv_hP}, and  taking  in~\eqref{unfold_conv_1_S} and~\eqref{unfold_conv_1_P} the limit as $\ve \to 0$ we obtain
\begin{equation}\label{unfold_conv_1_S_lim}
\begin{aligned}
- \big\langle {\theta}_S(y,h_S), \partial_t \phi_1  \big\rangle_{\Omega_T\times \hat S}
+   k_\Gamma \big\langle h_S-  h_P, \phi_1  \big\rangle_{\Omega_T\times \Gamma_P} + \big \langle f(h_S), \phi_1 \big\rangle_{\Gamma_{0, T}\times \hat S}\\
 + \big\langle K_S(y, h_{S})(\nabla h_{S} + \nabla_{y,0} h_{S,1} + {e}_3), \nabla \phi_1+ \nabla_{y,0} \phi_2\big \rangle_{\Omega_T \times \hat S} =  0 ,
\end{aligned}
\end{equation}
\begin{equation}\label{unfold_conv_1_P_lim}
\begin{aligned}
- \big\langle {\theta}_P(h_P), \partial_t \psi \big\rangle_{\Omega_T\times Y_P}  + \big\langle \tilde  K_P(h_{P})(\partial_{x_3} h_{P} + 1), \partial_{x_3} \psi \big \rangle_{\Omega_T\times Y_P}  \\
+  k_\Gamma \big \langle h_P- h_S, \psi \big\rangle_{\Omega_T\times \Gamma_P} +  \big\langle \mathcal{T}_\text{pot},\psi \big\rangle_{\Gamma_{0,T}\times Y_P} & = 0.
\end{aligned}
\end{equation}
 Notice that uniform boundedness and continuity of $K_J$, for $J=P,R,B$, and convergence a.e.~of $\mT^\ve(h_S^\ve)$ and $\mT^\ve(h_P^\ve)$, together with the Lebesgue dominated convergence theorem,  imply strong convergence   $K_S(y,\mT^\ve(h_{S}^\ve)) \to K_S(y, h_S)$ in $L^p((0,T)\times \Omega \times \hat S)$ and $K_P(\mT^\ve(h_{P}^\ve)) \to K_P(h_P)$  in
$L^p((0,T)\times \Omega \times Y_P)$, for any $1<p<\infty$.  To show  the convergence of $f(\mT^\ve(h_S^\ve))$ we consider
\begin{equation}\label{estim:conv_fhS}
\begin{aligned}
 &\|\theta_S(\cdot, \mT^\ve(h_S^\ve)) - \theta_S(\cdot, h_S)\|^2_{L^2(\Gamma_{0, T}\times \hat S)} \! \leq C \! \!\!\!\sum_{J=R,B}\!\!\! \!\||\theta_J(h_S^\ve) - \theta_J(h_S)|^2\|_{L^1(\Gamma^\ve_{J, 0, T})}
\\ & \; \leq C\!\!\! \sum_{J=R,B}\!\!\Big[\||\theta_J(h_S^\ve) - \theta_J(h_S)|^2\|_{L^1(J^\ve_T)}
+  \|I_\ve\nabla(\theta_J(h_S^\ve) - \theta_J(h_S))^2\|_{L^1(J^\ve_T)} \Big]\\
& \; \leq  C \Big[\|\theta_S(\cdot, h_S^\ve) - \theta_S(\cdot, h_S)\|^2_{L^2(S^\ve_T)}  \\
&\qquad +   \|\theta_S(\cdot, h_S^\ve) - \theta_S(\cdot, h_S)\|_{L^2(S^\ve_T)}\big(\|\nabla h_S^\ve\|_{L^2(S^\ve_T)} + \|\nabla h_S\|_{L^2(S^\ve_T)}\big) \Big].
\end{aligned}
\end{equation}
Here we used the trace theorem   and  the standard scaling argument to obtain
$$
\|u\|_{L^1(\Gamma^\ve_{J,0})} \leq C \big( \|u\|_{L^1(J^\ve)} +  \|I_\ve \nabla u\|_{L^1(J^\ve)} \big), \quad \text{ for } \; u \in W^{1,1}(J^\ve), \; \; J = R,B.
$$
Then using in \eqref{estim:conv_fhS} the strong convergence of  $\theta_S(y, \mT^\ve(h_S^\ve))$ in $L^2((0,T)\times \Omega\times \hat S)$, see Lemma~\ref{lemma:two-scale_convergence_theta_S}, we obtain the strong convergence of $\theta_S(y, \mT^\ve(h_S^\ve))$ in $L^2((0,T)\times \Gamma_0\times \hat S)$. The monotonicity and continuity  of $\theta_S(y, \cdot)$ ensure  $\mT^\ve(h_S^\ve) \to h_S$ a.e.~in $(0,T)\times \Gamma_0\times \hat S$ and then the continuity and boundedness of $f$ imply  $f(\mT^\ve(h_S^\ve))\to f(h_S)$  in  $L^2((0,T)\times \Gamma_0\times \hat S)$.

Using the standard arguments, see e.g.~\cite{Allaire_1992, cioranescu2012homogenization}, from \eqref{unfold_conv_1_S_lim} by considering $\phi_1 =0$ we obtain the unit cell problems~\eqref{unit_cell_prob_macro} and the formula for $K_{S, {\rm hom}}$. Then, considering \eqref{unfold_conv_1_S_lim}, with $\phi_2 =0$ and first $\phi_1 \in C^1_0((0,T)\times \Omega)$ and then $\phi_1 \in C_0^1(0, T; C^1_{\Gamma_{L_3}}(\Omega))$,   and \eqref{unfold_conv_1_P_lim}  first with  $\psi \in C^1_0((0,T)\times \Omega)$ and then $\psi \in C_0^1(0, T; C^1_{\Gamma_{L_3}}(\Omega))$ yields the macroscopic equations \eqref{macro_model}. Considering $\phi_2 =0$ and $\phi_1, \psi \in C^1([0,T]; C^1_0 (\Omega))$, with $\phi_1 (T) =\psi(T) =0$, as test functions   in the weak formulation of~\eqref{richards_soil},~\eqref{bcs_soil}, and~\eqref{richards_root},~\eqref{bcs_root} respectively, and using  monotonicity of $\theta_S^\ast$ and $\theta_P$ imply that $h_S$ and $h_P$ satisfy the corresponding initial conditions.

Using Assumption~\ref{assumption}, together with the Lipschitz continuity of $K_B$, $K_R$ and $K_P$, and $f$ being non-decreasing, in the similar way as in the proof of Theorem~\ref{theorem_microscopic_models_uniqueness}, we obtain the uniqueness of solutions of the macroscopic model~\eqref{macro_model}.
 More specifically, following the same calculations as in Lemma~\ref{lemma_uniqueness_lemma_1} and considering $\sigma_\delta^\pm(h_S-v^0)$ and $\sigma_\delta^\pm(h_P-w^0)$, for $v_0 \in V(\Omega)$ and $w_0 - a \in U(\Omega)$ and  $\sigma^\pm_\delta$ defined in \eqref{our_eta_prime}, as test functions for macroscopic equations~\eqref{macro_model}, we obtain
$$
\begin{aligned}
-\int_0^T\!\!  \Big\langle\eta_{S,\delta}^+(h_{S},v^0), \frac{d\kappa}{dt}\Big\rangle_{\Omega} dt +
\int_0^T\!\!  \Big[\big\langle K_{S, \rm hom}(h_S)(\nabla h_{S} + {e}_3),\! \nabla \sigma_\delta^+(h_S-v^0) \big \rangle_{\Omega}
\\
\qquad  +   k_\Gamma \vartheta_\Gamma \big\langle h_S-h_P, \sigma_\delta^+(h_S-v^0)  \big \rangle_{\Omega} +
\big\langle \vartheta_S f(h_S), \sigma_\delta^+(h_S-v^0) \big\rangle_{\Gamma_0} \Big] \kappa (t) dt \leq 0, \\
-\int_0^T\!\! 	\Big\langle \vartheta_P \eta_{P,\delta}^+(h_{P},w^0), \frac{d\kappa}{dt}\Big\rangle_{\Omega} \! +
\int_0^T\!\! \Big[\big \langle \vartheta_P \tilde K_P(h_P)(\partial_{x_3} h_{P} + 1),  \partial_{x_3} \sigma_\delta^+(h_P-w^0)\big \rangle_{\Omega} \\
+   k_\Gamma \vartheta_{\Gamma} \big\langle h_P-h_S,  \sigma_\delta^+(h_P-w^0) \big \rangle_{\Omega}	+ \big \langle \vartheta_P \mathcal{T}_\text{pot},\sigma_\delta^+(h_P-w^0) \big\rangle_{\Gamma_0 }\Big] \kappa (t) dt \leq 0,
\end{aligned}
$$
for $\kappa \in C^\infty_0(0,T)$,
where $\eta^\pm_{S, \delta}(h_S, v^0) =  \int_{v^0}^{h_S}\sigma_\delta^\pm(z-v^0)\big[ {\theta}_R^{'}(z) \vartheta_R   +   {\theta}_B^{'}(z)\vartheta_B \big]dz$ and  $\eta^\pm_{P,\delta}(h_P,w^0) = \int_{w^0}^{h_P}\sigma_\delta^\pm(z-w^0) {\theta}_P^{'}(z)dz$, and $\vartheta_P=|Y_P|/|Y|$.
The same inequality holds  with $\eta_{S,\delta}^-(h_{S},v^0)$ and  $\eta_{P,\delta}^-(h_{P},w^0)$ instead of $\eta_{S,\delta}^+(h_{S},v^0)$ and  $\eta_{P,\delta}^+(h_{P},w^0)$,  and  with $\sigma_\delta^-(h_S-v^0)$ and  $\sigma_\delta^-(h_P-w^0)$ 	instead of $\sigma_\delta^+(h_S-v^0)$ and  $\sigma_\delta^+(h_P-w^0)$. Applying the time variable doubling approach and following the same calculations as in the proof of Theorem~\ref{theorem_microscopic_models_uniqueness}, yield the uniqueness result for problem~\eqref{macro_model}.  Notice that  in both inequalities for $h_P$ and $h_S$ the terms involving spatial derivatives are estimated in the same way as in~\eqref{inequality_uniqueness_lower_bound_flux_term_double_time}.  Consideration of
cases (i)~$h_{S,1} - h_{S,2} \leq 0$,  (ii)~$h_{P,1} - h_{P,2} \leq 0$, (iii)~$h_{S,1} - h_{S,2} \geq \delta$,  $h_{P,1} - h_{P,2} \geq \delta$,  (iv)~$h_{S,1} - h_{S,2} \geq \delta$, $0<h_{P,1} - h_{P,2} <\delta$,
(v)~$\delta > h_{S,1} - h_{S,2}> h_{P,1} - h_{P,2} >0$, and (vi)~$h_{P,1} - h_{P,2}> h_{S,1} - h_{S,2} >0$, where $(h_{S,1}, h_{P,1})$ and $(h_{S,2}, h_{P,2})$ are two solutions of~\eqref{macro_model}, in the same way as in the proof of Theorem~\ref{theorem_microscopic_models_uniqueness},  implies  $\big\langle (h_{S,1} - h_{S,2})-(h_{P,1} - h_{P,2}), \sigma_\delta^+(h_{S,1}-h_{S,2})  - \sigma_\delta^+(h_{P,1}-h_{P,2}) \big \rangle_{\Omega}\geq 0$. The non-decreasing assumption on $f$ ensures the non-negativity of the boundary integral.
 \end{proof}

\section{Conclusion}
In this work we considered multiscale modelling of root and soil water transport, accounting for root distribution in the soil and the differing hydraulic properties of plant tissue, the rhizosphere and bulk soil. Water uptake by roots was defined by the differences in pressure head across the soil-root interfaces. In the macroscopic model~\eqref{macro_model},  derived using homogenization techniques, the effects of plant root presence and varying soil hydraulic properties are incorporated via the water content and hydraulic conductivity functions that encode the parameters and structure from the microscale. Due to considered scaling in the microscopic model, motivated by experimental parameter values, the effective macroscopic functions incorporating microscopic properties can be computed in advance,
thus making the macroscopic model computationally efficient. The macroscopic water content functions and hydraulic conductivity matrix in~\eqref{macro_model} are related to the matrix function and factor that are used in~\cite{mair2022model}  to account for the proportion of the soil
occupied   by roots and model the preferential soil water flow that their presence induces.
Rigorous derivation of the macroscopic model~\eqref{macro_model} provides justification for the phenomenologically-defined hydraulic conductivity in the preferential flow model in~\cite{mair2022model}. Model~\eqref{macro_model} also incorporates root water
uptake and root influence on soil hydraulic properties at the macroscopic scale, taking
into account microscopic characteristics of the soil and plant tissue.  This provides a formulation that is more detailed and better suited to parametrisation than the phenomenological formulation  of~\cite{mair2022model,mair2023can}, whilst also being more computationally efficient
than the microscopic models~\cite{doussan1998modelling, doussan2006water}. Derivation of a macroscopic root water uptake
function from a micro-scale network model of root system hydraulic architecture was
addressed in~\cite{vanderborght2024combining}, but only a one-dimensional domain was considered and water content was assumed uniform in certain soil volumes. Numerical analysis and simulations of the
macroscopic model~\eqref{macro_model} and extension of the multiscale analysis to locally-periodic
root systems will be considered in the future research.


\section*{Acknowledgements}
AM was supported by the EPSRC Centre for Doctoral Training in Mathematical Modelling, Analysis \& Computation
(MAC-MIGS) funded by the UK Engineering and Physical Sciences Research Council (EPSRC) grant EP/S023291/1,
Heriot-Watt University, and the University of Edinburgh.

\section*{Appendix: Admissible functions for water content and hydraulic conductivity}
The function for soil water content~$\theta_S$ can be defined in accordance with the formulation of~\cite{van1980closed} and  extended for positive values of soil water pressure head
\begin{equation}\label{standard_soil_water_content_expression}
\theta_S(h_S)=\begin{cases}
\frac{\theta_{S,\text{sat}}-1}{(1+\lvert\alpha h_S\rvert^{n})^{m}} + 1, & \; h_S >0,\\
\theta_{S,\text{res}} + \frac{\theta_{S,\text{sat}}-\theta_{S,\text{res}}}{(1+\lvert\alpha h_S\rvert^{n})^{m}}, & \; h_S\leq 0.
\end{cases}
\end{equation}
The function $\theta_S$ has different parameter values for bulk soil and the rhizosphere. Water content~$\theta_P$ for root tissue can be defined using the models in~\citep{janott2011one,bittner2012functional} and extending for positive root tissue pressure heads in a similar way as~\eqref{standard_soil_water_content_expression}:
\begin{equation}\label{standard_root_water_content_expression}
\theta_P(h_P) = \begin{cases}
\frac{\theta_{P,\text{sat}}-1}{(1+\lvert\alpha h_P\rvert^{n})^{m}} + 1, & \;  h_P>0, \\
\Big(\theta_{P,\text{sat}} + \frac{h_P}{E}\Big), & \; h_{\text{ae}}\leq h_P < 0,\\
\Big(\theta_{P,\text{sat}} + \frac{h_{\text{ae}}}{E}\Big)\Big(\frac{h_P}{h_{\text{ae}}}\Big)^{-\lambda_P}, & \; h_P< h_{\text{ae}}.
\end{cases}
\end{equation}
Here~$\theta_{P,\text{sat}}$~$\big(\text{L}^3\text{L}^{-3}\big)$ is the root tissue porosity (or root xylem saturated water content), the air entry pressure head is~$h_{\text{ae}}$~$(\text{L})$, the root xylem elastic modulus is~$E$~$(L)$, $\lambda_P$ is the Brooks and Corey exponent,
and~$\alpha$,~$n>1$ and~$m=1-1/n$ are as in~\citep{van1980closed}.

A suitable expression for~$K_S$ can be defined by taking a regularisation of the Van Genuchten~\citep{van1980closed} formulation and extending it for positive pressure heads. The first step is to define the  regularisation of the water content function
$$
\theta_{S,\delta}(h_S) =
\begin{cases}
\frac{\theta_{S,\text{sat}}-1}{(1+\lvert\alpha h_S\rvert^{n})^{m}} + 1 - \frac{\delta}{2}, & \; h_S >0,\\
\theta_{S,\text{res}} + \frac{\theta_{S,\text{sat}}-(\theta_{S,\text{res}}+\delta)}{(1+\lvert\alpha h_S\rvert^{n})^{m}} + \frac{\delta}{2},
& \; h_S\leq 0,
\end{cases}
$$
where~$0<\delta<<\theta_{S,\text{sat}}-\theta_{S,\text{res}}$.
The regularised soil hydraulic conductivity $K_{S,\delta}(h_S)$ satisfying the conditions in Theorems~\ref{theorem_microscopic_models_existence} and~\ref{theorem_microscopic_models_uniqueness} is given as
$$
K_{S,\delta}(h_S) \!\!=\!\!
\begin{cases}
K_{S,\text{sat}}\Big[1-\frac{\delta}{2(\theta_{S,\text{sat}}-\theta_{S,\text{res}})}\Big]^l\Big[1-\Big[1-\Big[1-\frac{\delta}{2(\theta_{S,\text{sat}}-\theta_{S,\text{res}})}\Big]^{\frac{1}{m}}\Big]^{m}\Big]^2,& \!\!\!\! h_S > 0,\\
K_{S,\text{sat}}\Big[\frac{\theta_{S,\delta}(h_S)-\theta_{S,\text{res}}}{\theta_{S,\text{sat}}-\theta_{S,\text{res}}}\Big]^l\Big[1 - \Big[1-\Big[\frac{\theta_{S,\delta}(h_S)-\theta_{S,\text{res}}}{\theta_{S,\text{sat}}-\theta_{S,\text{res}}}\Big]^{\frac{1}{m}}\Big]^{m}\Big]^2,& \!\!\! \! h_S \leq 0.
\end{cases}
$$
In a similar way, we define a regularised expression for root tissue water content
$$
\theta_{P, \delta}(h_P) =
\begin{cases}
\frac{\theta_{P,\text{sat}}-1}{(1+ |\alpha h_P|^{n})^{m}} + 1, & \; h_P>0,\\
\Big(\theta_{P,\text{sat}} + \frac{h_P}{E}\Big), & \; h_{\text{ae}}\leq h_P < 0,\\
\Big(\theta_{P,\text{sat}} + \frac{h_{\text{ae}}}{E}\Big)\Big(\frac{h_P}{h_{\text{ae}}}\Big)^{-\lambda_P} + \delta(1-e^{h_P-h_{\text{ae}}}),
& \; h_P< h_{\text{ae}},
\end{cases}
$$
where~$\delta >0$, and a regularised version of the function~$K_{P}(h_P)$ given as
$$
K_{P, \delta}(h_P) \!\! =	\!\!	\begin{cases}
1, & \!\!\! \!h_P \geq 0,\\
\big(1+\frac{h_\text{ae}}{\theta_{P,\text{sat}} E}\big)^{\frac{2}{\lambda_P}+1} +
\Big[1-\big(1+\frac{h_\text{ae}}{\theta_{P,\text{sat}} E}\big)^{\frac{2}{\lambda_P}+1}\Big]\big[\frac{h_P-h_{\text{ae}}}{h_{\text{ae}}}\big]^2,
& \!\! \! \! h_{\text{ae}} \leq h_P < 0,\\
\Big[\big(1 + \frac{h_{\text{ae}}}{\theta_{P,\text{sat}} E}\big)\big(\frac{h_P}{h_{\text{ae}}}\big)^{-\lambda_P}+\frac{\delta}{\theta_{P,\text{sat}}}\big(1-e^{(h_P-h_{\text{ae}})}\big)\Big]^{\frac{2}{\lambda_P}+1}, & \!\!\!\! h_p < h_{\text{ae}}<0.
\end{cases}
$$
An admissible function~$f:\mathbb{R}\to\mathbb{R}$ for the water flux at the upper soil surface~$\Gamma_{S, 0}$ is
\begin{equation}\label{admissable_f}
\begin{aligned}
f(h_S)= \text{K}_\text{e}(h_S)\text{ET}_\text{o}-\text{P} + \text{RO}(h_S),
\end{aligned}
\end{equation}
where~$\text{ET}_\text{o}$ is the reference evapotranspiration~$(\text{L~T}^{-1})$, see~\citep{allen1998crop}, $\text{K}_\text{e}:\mathbb{R}\to[0,\infty)]$ controls the amount of evaporation from the soil surface, $\text{P}$~$(\text{L~T}^{-1})$ models the precipitation. The term~$\text{RO}$~$(\text{L~T}^{-1})$ incorporates runoff which occurs when precipitation lands on a saturated soil surface and cannot infiltrate downwards, and can be defined as
\begin{equation}\label{model: runoff function}
\begin{aligned}
\text{RO}(h,t) = \text{P}(t)\Big(1+\frac{\exp(a(\theta_S(h_S)-\theta_{S, \text{sat}})) - 1}{\exp(a(\theta_S(h_S)-\theta_{S, \text{sat}})) + 1}\Big),
\end{aligned}
\end{equation}
with the constant~$a>0$ set to a large enough value so that if the surface is fully saturated, $\theta_S = \theta_{S, \text{sat}}$, then~$\text{RO} = \text{P}$. In this work, the rate of precipitation $\text{P}$ is assumed to be constant. The evaporation function~$\text{K}_\text{e}$ can be defined  as in~\citep{allen1998crop}:
\begin{equation}\label{evaporation_function}
\begin{aligned}
\text{K}_\text{e}(h_S)=\min\{\text{K}_\text{r}(h_S)(\text{K}_\text{c max}-\text{K}_\text{cb}), \text{f}_\text{ew}\text{K}_\text{c max}\},
\end{aligned}
\end{equation}
where~$\text{K}_\text{c max}>0$ is a constant such that~$\text{K}_\text{c max}\text{ET}_\text{o}$ is the upper limit on evapotranspiration from the soil surface, and~$f_\text{ew}>0$ represents the fraction of the soil surface that is exposed and wetted  with  $
f_\text{ew}=\min\{1-f_\text{c}, f_\text{w}\}$, where~$f_\text{c}$ is the fraction of the soil surface that is covered by vegetation and~$f_\text{w}$ the average fraction that is wetted during precipitation.
The soil evaporation reduction coefficient~$\text{K}_\text{r}$ depends on soil water content and can be formulated as
$$
\text{K}_\text{r}(h_S)=
\begin{cases}
0, &\quad  \theta_S(h_S) \leq 0.5\theta_{\text{wp}},\\
\frac{\theta_S(h_S)-0.5\theta_{\text{wp}}}{\theta_\text{fc}-0.5\theta_\text{wp}}, &\quad  0.5\theta_\text{wp}<\theta(h)<\theta_\text{fc},\\
1, & \quad \theta_S(h_S)\geq\theta_\text{fc},
\end{cases}
$$
see~\citep{mair2023can}, where~$\theta_{\text{wp}}$ and~$\theta_{\text{fc}}$ are the water contents at wilting point and field capacity respectively~\cite{allen1998crop}.

\bibliographystyle{siamplain}
\bibliography{homogenization_bibliography}

\end{document}